\documentclass{amsart}
\usepackage{amsmath}
\usepackage{paralist}
\usepackage{amsfonts}
\usepackage{amssymb}
\usepackage{amsthm}
\usepackage{amscd}
\usepackage{amsrefs}
\usepackage{float}
\usepackage{tikz}
\usepackage{graphicx}
\usepackage[colorlinks=true]{hyperref}
\hypersetup{urlcolor=blue, citecolor=red}
\usepackage{hyperref}

\newtheorem{theorem}{Theorem}[section]

\newtheorem{lemma}[theorem]{Lemma}
\newtheorem{proposition}[theorem]{Proposition}

\theoremstyle{definition}
\newtheorem{definition}[theorem]{Definition}
\newtheorem{remark}[theorem]{Remark}

\newcommand{\T}{\mathbb{T}}
\newcommand{\R}{\mathbb{R}}
\newcommand{\C}{\mathbb{C}}
\newcommand{\Z}{\mathbb{Z}}
\newcommand{\N}{\mathbb{N}}

\begin{document}

\title[Periodic solutions to the modified Korteweg-de Vries equation]{On the existence of periodic solutions to the modified Korteweg-de Vries equation below $H^{1/2}(\mathbb{T})$}

\author{Robert Schippa}
\address{Fakult\"at f\"ur Mathematik, Universit\"at Bielefeld, Postfach 10 01 31, 33501 Bielefeld, Germany}
\email{robert.schippa@uni-bielefeld.de}

\begin{abstract}
Existence and a priori estimates for real-valued periodic solutions to the modified Korteweg-de Vries equation with initial data in $H^s$ are established for $s>0$. The short-time Fourier restriction norm method is employed to overcome the derivative loss. Further, non-existence of solutions below $L^2$ is proved conditional upon conjectured linear Strichartz estimates.
\end{abstract}

\maketitle
\bigskip

\section{Introduction}
\label{section:introduction}
In this paper we consider the Cauchy problem for the modified Korteweg-de Vries (mKdV) equation
\begin{equation}
\label{eq:mKdV}
\left\{\begin{array}{cl}
\partial_t u + \partial_{xxx} u &= \pm u^2 \partial_x u ,  \; (t,x) \in \mathbb{R} \times \mathbb{T}, \\
u(0) &= u_0, \end{array} \right.
\end{equation}
posed on the circle $\mathbb{T} = \mathbb{R}/(2 \pi \mathbb{Z})$ with real-valued initial data $u_0$.
On the real line, there is the scaling symmetry
\begin{equation*}
u(t,x) \rightarrow \lambda u(\lambda^3 t, \lambda x), \; u_0(x) \rightarrow \lambda u_0(\lambda x),
\end{equation*}
which leads to the scaling invariant homogeneous Sobolev space $\dot{H}^{-1/2}(\mathbb{R})$.\\
The energy is given by
\begin{equation}
\label{eq:energies}
E[u] = \int_{\T} \frac{(\partial_x u)^2}{2} \pm \frac{u^4}{12},
\end{equation}
where the signs from \eqref{eq:mKdV} match the signs in \eqref{eq:energies}. Hence, the positive sign gives rise to the defocusing and the negative sign gives rise to the focusing modified Korteweg-de Vries equation. The mKdV equation is closely related to the classical Korteweg-de Vries equation
\begin{equation}
\label{eq:KdV}
\left\{\begin{array}{cl}
\partial_t u + \partial_{xxx} u &=  \partial_x (u^2)/2, \quad (t,x) \in \mathbb{R} \times \mathbb{T}, \\
u(0) &= u_0.\end{array} \right.
\end{equation}
Both Cauchy problems were thoroughly investigated, and the list of literature is extensive. Below, we do not give a complete description of previous works on \eqref{eq:mKdV} or \eqref{eq:KdV}, but rather an excerpt of work more closely related to this article. The reader is also referred to the list of literature therein.

First well-posedness results were established by energy arguments (\cite{BonaSmith1975,AbdelouhabBonaFellandSaut1989}). Beyond, Bourgain proved in \cite{Bourgain1993FourierTransformRestrictionPhenomenaII} that the Cauchy problem for the mKdV equation is analytically locally well-posed in $H^s(\mathbb{T})$ for $s \geq 1/2$ and globally well-posed in $H^s(\mathbb{T})$ for $s \geq 1$ in Fourier restriction spaces.\\
Making use of the $I$-method, it was proved in \cite{CollianderKeelStaffilaniTakaokaTao2003} by Colliander et al. that \eqref{eq:mKdV} is globally well-posed in $H^s(\mathbb{T})$ for $s \geq 1/2$.\\
It is well-known that one can map solutions from the defocusing mKdV equation to solutions to the KdV equation employing the Miura transform (cf. \cite{Miura1968,KappelerTopalov2005}). With the KdV equation being completely integrable, one also finds the mKdV equation to be completely integrable.

We stress that although several of the symmetries of the mKdV equation are certainly used in the proof of the main result, in particular that real-valued initial data give rise to real-valued solutions, the method does not depend on complete integrability. Additionally, we discuss regularity and existence of solutions to the KdV-mKdV-equation (cf. \cite{Molinet2012})
\begin{equation}
\label{eq:KdVmKdVEquation}
\left\{\begin{array}{cl}
\partial_t u + \partial_{xxx} u &= u\partial_x u \pm u^2 \partial_x u, \quad (t,x) \in \R \times \T \\
u(0) &= u_0 \in H^s(\T)
\end{array}\right.
\end{equation}
and the following mKdV-mKdV-system
\begin{equation}
\label{eq:mKdVmKdVSystem}
\left\{\begin{array}{cl}
\partial_t u + \partial_{xxx} u &= \partial_x (u v^2), \quad (t,x) \in \R \times \T \\
\partial_t v + \partial_{xxx} v &= \partial_x (v u^2)
\end{array}\right.
\end{equation}
with $(u(0),v(0)) = (u_0,v_0) \in H^s(\T) \times H^s(\T)$. The analysis gives the same a priori estimates like for the mKdV-equation.

Further, the strategy can be adapted to consider generalized KdV equations (cf. \cite{Staffilani1997GKDV,CollianderKeelStaffilaniTakaokaTao2004}) like
\begin{equation*}
 \partial_t u + \partial_{xxx} u = \pm u^{k-1} \partial_x u, \; (t,x) \in \mathbb{R} \times \mathbb{T}
\end{equation*}
or dispersion generalized equations like
\begin{equation*}
\partial_t u + \partial_x D_x^{a-1} u = \pm u^2 \partial_x u, \; (t,x) \in \mathbb{R} \times \T,
\end{equation*}
which are no longer amenable to inverse scattering techniques.

Exploiting the integrability properties and the inverse scattering transform, Kappeler and Topalov showed \eqref{eq:mKdV} in the defocusing case to be globally well-posed in $L^2(\mathbb{T})$ (cf. \cite{KappelerTopalov2005})
with a notion of solutions defined through smooth approximations. From Sobolev embedding one finds that these solutions satisfy the mKdV equation in the sense of generalized functions as soon as $s \geq 1/6$. The result was recently extended to the real line case and simplified in \cite{KillipVisan2018}.\\
Unconditional well-posedness of the mKdV equation by means of normal form reduction was shown in \cite{KwonOh2012} for $s \geq 1/2$.

Since the mKdV equation is completely integrable, there is an infinite number of conserved quantities along the flow. In addition to the conservation of energy, we record the conservation of mass for real-valued solutions, i.e.,
\begin{equation*}
\int_{\mathbb{T}} u^2(t,x) dx = \int_{\mathbb{T}} u_0^2(x) dx
\end{equation*}
because this provides us with an $L^2$-a priori estimate $\sup_{t \in \mathbb{R}} \Vert u(t) \Vert_{L^2(\mathbb{T})} \lesssim \Vert u_0 \Vert_{L^2(\mathbb{T})}$ for smooth solutions.\\
It is known that the data-to-solution map fails to be $C^3$ below $s<1/2$ (cf. \cite{Bourgain1997MeasuresInitialData}) and even fails to be uniformly continuous (cf. \cite{BurqGerardTzvetkov2002,ChristCollianderTao2003}) because of the resonant term on the diagonal.

Non-diagonal resonant interactions can be removed by changing to the renormalized modified Korteweg-de Vries equation
\begin{equation}
\label{eq:renormalizedmKdV}
\partial_t u + \partial_{xxx} u = (u^2 - \frac{1}{2 \pi} \int_{\mathbb{T}} u^2) \partial_x u = \mathfrak{N}(u).
\end{equation}
The solution to \eqref{eq:renormalizedmKdV} is given in terms of the solution to \eqref{eq:mKdV} by
\begin{equation}
\label{eq:relationRenormalizedmKdV}
v(t,x) = u(t,x - C( \int_0^t \int_{\mathbb{T}} u^2(t^\prime,x^\prime) dx^\prime dt^\prime )) = u(t,x-Ct \Vert u_0 \Vert_{L^2}^2).
\end{equation}
The norm of the solution to \eqref{eq:renormalizedmKdV} in non-negative Sobolev spaces equals the one of the solution to \eqref{eq:mKdV}, and most of the well-posedness results were shown for the renormalized mKdV equation as one can see that removing the off-diagonal interactions introduces a drift term governed by the $L^2$-norm.\\
In negative Sobolev spaces the nonlinear interaction $\mathfrak{N}$ for \eqref{eq:renormalizedmKdV} is defined in Fourier variables, see below. For the technical reason of having the non-diagonal interactions removed, we will also work on the renormalized version \eqref{eq:renormalizedmKdV}, but according to the above considerations, the Cauchy problems are essentially equivalent for regularities above $L^2$.

Not hinging on complete integrability, but instead employing a nonlinear modification of the Fourier restriction spaces, in \cite{NakanishiTakaokaTsutsumi2010} Nakanishi et al. showed local well-posedness of the Cauchy problem associated with \eqref{eq:mKdV} (i.e. continuous dependence on the initial data) for $s > 1/3$ and a priori estimates for $s>1/4$ (see also the previous work \cite{TakaokaTsutsumi2004} by Takaoka and Tsutsumi).\\
Combining the normal form approach from \cite{KwonOh2012} and the nonlinear ansatz from \cite{NakanishiTakaokaTsutsumi2010}, Molinet et al. proved unconditional well-posedness for $s \geq 1/3$ in \cite{MolinetPilodVento2019}.

In another work \cite{Molinet2012} by Molinet, it was shown that the Kappeler-Topalov solutions satisfy the defocusing mKdV equation in $L^2(\mathbb{T})$ in the distributional sense.\\
In the focusing case, relying on the conservation of mass and using short-time Fourier restriction, was shown the existence of global distributional solutions in $C_w(\R;L^2(\T))$, that means the solutions are continuous curves in $L^2(\T)$ endowed with the weak topology.\\
 In \cite{Molinet2012} was also proved that the data-to-solution map fails to be continuous from $L^2(\mathbb{T})$ to $\mathcal{D}^\prime([0,T])$ for non-constant initial data $u_0 \in H^{\infty}(\mathbb{T})$. Here, also short-time Fourier restriction norm spaces were used to control the cubic derivative interaction. We will revisit the analysis and see that one can control the nonlinear interaction below $L^2$ for suitable frequency dependent time localization.
 
On the real line, \eqref{eq:mKdV} is better behaved than on the torus because of stronger dispersive effects. In \cite{KenigPonceVega1993} (see also \cite{Tao2001}) was shown that \eqref{eq:mKdV} is locally well-posed for $s>1/4$ by a Picard iteration scheme in a resolution space capturing the dispersive effects. Global well-posedness for $s>1/4$ was also shown in \cite{CollianderKeelStaffilaniTakaokaTao2003}.\\
In \cite{ChristHolmerTataru2012} Christ et al. showed a priori estimates for smooth solutions for $-1/8<s \leq 1/4$ making use of the short-time Fourier restriction spaces.

When we refer to existence of solutions in the following, we refer to the existence of a data-to-solution mapping $S:H^s \rightarrow C([-T,T],H^s)$ where $T=T(\Vert u_0 \Vert_{H^s})>0$ with the following properties:
\begin{enumerate}
\item[(i)] $S(u_0)$ satisfies the equation in the distributional sense and $S(u_0)(0) = u_0$.
\item[(ii)] There exists a sequence of smooth global solutions $(u_n)$ such that $u_n \rightarrow S(u_0)$ in $C([-T,T],H^s)$ as $n \rightarrow \infty$.
\end{enumerate}
This notion was introduced in \cite{GuoOh2018} to discuss existence of solutions to the nonlinear Schr\"odinger equation on the circle for low regularities.\\
We recall why the second property is natural for two reasons following \cite{GuoOh2018}. Local well-posedness requires continuity of the data-to-solution map, but also from a practical point of view the construction of solutions typically requires at least one approximating sequence of smooth global solutions.

Main purpose of this article is to show the existence of solutions and a priori estimates below $H^{1/2}(\mathbb{T})$ up to $L^2(\mathbb{T})$ relying on localization in time of the Fourier restriction spaces. The frequency dependent localization in time introduces extra smoothing, which allows us to overcome the derivative loss for low regularities.

Essentially\footnote{For the actually more involved energy estimate see Section \ref{section:EnergyEstimates}.}, we will show the following three estimates for $T \in (0,1]$ and $s>0$:
\begin{equation*}
\left\{\begin{array}{cl}
\Vert u \Vert_{F^{s,\alpha}(T)} &\lesssim \Vert u \Vert_{E^s(T)} + \Vert \mathfrak{N}(u) \Vert_{N^{s,\alpha}(T)} \\
\Vert \mathfrak{N}(u) \Vert_{N^{s,\alpha}(T)} &\lesssim T^\theta \Vert u \Vert_{F^{s,\alpha}(T)}^3 \\
\Vert u \Vert_{E^s(T)}^2 &\lesssim \Vert u_0 \Vert_{H^s}^2 + T^\theta \Vert u \Vert_{F^{s,\alpha}(T)}^6
\end{array} \right.
\end{equation*}
We compare this set of estimates to estimates for the classical Fourier restriction norms.\\
Let $G$ denote in the following the nonlinearity of the dispersive equation under consideration (for more details on the notation, see e.g. \cite{Tao2006}). The first estimate relates to the $X^{s,b}$-energy estimate 
\begin{equation*}
\Vert \eta(t) u \Vert_{X^{s,b}} \lesssim_{b,\eta} \Vert u_0 \Vert_{H^s} + \Vert G(u) \Vert_{X^{s,b-1}} \quad (b>1/2).
\end{equation*}
Subsequently, one has to prove a nonlinear estimate 
\begin{equation*}
\Vert G(u) \Vert_{X^{s,b-1}} \lesssim g(\Vert u \Vert_{X^{s,b}}),
\end{equation*}
which is the classical analog of the second estimate from above.

The third estimate has no analog in classical $X^{s,b}$-spaces. This is due to the fact that performing a frequency dependent time localization only allows one to estimate the short-time Fourier restriction norm $F^{s,\alpha}(T)$ in terms of a short-time norm $N^{s,\alpha}(T)$ for the nonlinearity and an energy norm $E^s(T)$, which distinguishes dyadic frequency ranges. Here, $\alpha$ governs the ratio of time localization and frequency size.

 Consequently, one has to propagate the energy norm in terms of the short-time Fourier restriction norm. With the above set of estimates at disposal, bootstrap and compactness arguments allow us to prove the following theorem.
\begin{theorem}
\label{thm:localExistence}
Let $s>0$. Given $u_0 \in H^s(\mathbb{T})$, there is a function $T=T(\Vert u_0 \Vert_{H^s})$ so that there exists a local solution $u \in C([-T,T],H^s(\T))$ to \eqref{eq:mKdV}. Furthermore, we find the a priori estimate
\begin{equation*}
\sup_{t \in [-T,T]} \Vert u(t) \Vert_{H^s} \leq C \Vert u_0 \Vert_{H^s}
\end{equation*}
to hold.
\end{theorem}
There is also the recent work \cite{KillipVisanZhang2018} by Killip et al. relying on complete integrability, where a priori estimates for smooth periodic initial data are shown, too.\\
For a solution to \eqref{eq:mKdV} with smooth initial data $u_0$, the a priori estimate
\begin{equation*}
\Vert u(t) \Vert_{H^s(\T)} \lesssim \Vert u_0 \Vert_{H^s} (1+\Vert u_0 \Vert_{H^s}^2)^{\frac{|s|}{1-2|s|}}
\end{equation*}
is proved in \cite{KillipVisanZhang2018} for $-1/2<s<1/2$. By means of the transformation \eqref{eq:relationRenormalizedmKdV}, the a priori estimate extends to smooth solutions to \eqref{eq:renormalizedmKdV}.\\
Notably, in \cite{KillipVisanZhang2018} are also proved a priori estimates for smooth solutions to the cubic nonlinear Schr\"odinger equation
\begin{equation}
\label{eq:cubicNLS}
\left\{\begin{array}{cl}
i \partial_t u + \partial_{xx} u &= \pm |u|^2 u ,  \; (t,x) \in \mathbb{R} \times \mathbb{T}, \\
u(0) &= u_0 \in H^s(\mathbb{T}), \end{array} \right.
\end{equation}
in the same range $-1/2<s<1/2$. However, in \cite{GuoOh2018} it was shown that because the data-to-solution mapping can be constructed with the aid of compactness arguments for a renormalized version of \eqref{eq:cubicNLS} for $-1/8<s<0$, it can not exist for \eqref{eq:cubicNLS}.\\
In the context of Fourier Lebesgue spaces, which scale like negative Sobolev spaces, this program was carried out for \eqref{eq:mKdV} in \cite{KappelerMolnar2017} using on complete integrability.

Another purpose of this work is to point out the critical interactions, which require further comprehension, to clarify existence in negative Sobolev spaces. For the non-linear estimate we shall see that localizing time higher than reciprocal to the frequency size allows us to control the renormalized nonlinear interaction for negative Sobolev regularities.

The situation for the energy estimate is more delicate as the critical interactions in the energy estimate occur at small second resonance. These are the interactions we can not estimate below $L^2$ in this work without the currently unproved $L^6_{t,x}$-Strichartz estimate
\begin{equation*}
\Vert u \Vert_{L^8_{t,x}(\mathbb{R} \times \T)} \lesssim \Vert u \Vert_{X_{Airy}^{0+,4/9+}}.
\end{equation*}
The essentially sharp above display would follow from
\begin{equation}
\label{eq:L8Hypothesis}
\Vert u \Vert_{L^8_{t,x}(\mathbb{R} \times \T)} \lesssim \Vert u \Vert_{X_{Airy}^{0+,1/2+}},
\end{equation}
which was conjectured in \cite{Bourgain1993FourierTransformRestrictionPhenomenaII}. Although there has been substantial progress on Strichartz estimates on tori (cf. \cite{HuLi2013}), \eqref{eq:L8Hypothesis} seems to be out of reach at the moment. We refer to Subsection \ref{subsection:EnergyEstimatesNegativeSobolevSpaces} for a more detailed discussion.
\begin{theorem}
\label{thm:nonExistence}
Suppose that \eqref{eq:L8Hypothesis} holds.\\
Then, there is $s^\prime < 0$ so that for $s^\prime<s<0$ there exists $T=T(\Vert u_0 \Vert_{H^s})$ such that there exists a local solution $u \in C([-T,T],H^s(\T))$ to \eqref{eq:renormalizedmKdV}, and we find the a priori estimate
\begin{equation*}
\sup_{t \in [-T,T]} \Vert u(t) \Vert_{H^s} \leq C \Vert u_0 \Vert_{H^s}
\end{equation*}
to hold.\\
Furthermore, solutions to \eqref{eq:mKdV} do not exist for $s^\prime < s < 0$.
\end{theorem}
As pointed out above, the global well-posedness result from \cite{KappelerTopalov2005} exceeds Theorem \ref{thm:localExistence} in the defocusing case.\\
However, the analysis gives the same regularity results for related non-integrable models.
\begin{theorem}
\label{thm:localExistenceExtension}
Let $s>0$. Given $u_0 \in H^s(\mathbb{T})$, there is a function $T=T(\Vert u_0 \Vert_{H^s})$ so that there exists a local solution $u \in C([-T,T],H^s(\T))$ to \eqref{eq:KdVmKdVEquation}. Furthermore, we find the a priori estimate
\begin{equation*}
\sup_{t \in [-T,T]} \Vert u(t) \Vert_{H^s} \leq C \Vert u_0 \Vert_{H^s}
\end{equation*}
to hold. The respective existence and regularity assertions are also true for \eqref{eq:mKdVmKdVSystem}.
\end{theorem}
The article is structured as follows: In Section \ref{section:Notation} we introduce notation and state basic estimates for short-time Fourier restriction spaces. In Section \ref{section:Proof} we finish the proof of the a priori estimates relying on a short-time trilinear estimate from Section \ref{section:shorttimeTrilinearEstimates} and energy estimates from Section \ref{section:EnergyEstimates}. Multilinear estimates to prove the short-time trilinear estimate are discussed in Section \ref{section:MultilinearEstimates}. In Section \ref{section:ExtensionFurtherModels} existence and regularity of solutions to the KdV-mKdV-equation and the mKdV-mKdV-system are discussed.
\section{Notation and Basic Properties of Function Spaces}
\label{section:Notation}
Most of the correspondent estimates on the real line of estimates below can already be found in the seminal paper \cite{IonescuKenigTataru2008} by Ionescu et al., where short-time Fourier restriction spaces were introduced. Hence, we omit most of the proofs and mainly record the estimates, which will be used later. We point out that the idea of carrying out the analysis on small frequency-dependent time intervals has been utilized in works (cf. \cite{KochTataru2007,KochTzvetkov2005,ChristCollianderTao2008}) predating \cite{IonescuKenigTataru2008}.

Let $\eta_0: \mathbb{R} \rightarrow [0,1]$ denote an even smooth function, $\mathrm{supp}(\eta_0) \subseteq [-8/5,8/5]$, $\eta_0 \equiv 1$ on $[-5/4,5/4]$. For $k \in \mathbb{N}$ we set $\eta_k(\tau) = \eta_0(\tau/2^k)-\eta_0(\tau/2^{k-1})$. We set $\eta_{\leq m} = \sum_{j=0}^m \eta_j$ for $m \in \mathbb{N}$ and set $\mathbb{N}_0 = \mathbb{N} \cup \{0\}$.\\
We denote unions of intervals $I_n = \left\{\xi \in \mathbb{R} \; | \; |\xi| \in [N, 2N-1] \right\}, \;  N = 2^n, \; n \in \mathbb{N}_0$ and $I_{\leq 0} = [-1,1]$. The intervals $I_{\leq 0}$ and $(I_n)_{n \in \mathbb{N}_0}$ partition frequency space. We usually denote dyadic numbers by capital letters $N,K,J$ and their binary logarithm by $n,k,j$.

We write for the Fourier transform
\begin{equation*}
\mathcal{F}_xf(n) = \hat{f}(n) = \int_{\mathbb{T}} f(x) e^{-ix n} dx, \quad n \in \Z, \quad f \in C^{\infty}(\mathbb{T}),
\end{equation*}
which is extended on $L^2(\mathbb{T})$ in the usual way.\\
We also consider the Fourier transform in space and time
\begin{equation*}
\mathcal{F}_{t,x}f(\tau,n) = \tilde{f}(\tau,n) = \int_{\mathbb{R}} \int_{\mathbb{T}} f(t,x) e^{-it \tau} e^{-ix n} dx dt, \quad f: \mathbb{R} \times \mathbb{T} \rightarrow \mathbb{C}.
\end{equation*}

For the Littlewood-Paley projector onto frequencies of order $2^k, \; k \in \mathbb{N}_0$, we write $P_k:L^2(\mathbb{T}) \rightarrow L^2(\mathbb{T})$, that is $(P_k u)\hat{\,}(\xi) = 1_{I_k}(\xi) \hat{u}(\xi)$. The dispersion relation for the Airy equation is given by $\omega(\xi) = \xi^3$.\\
We set for $k \in \mathbb{N}_0$ and $j \in \mathbb{N}_0$
\begin{equation*}
\begin{split}
D_{k,j} = \{(\xi,\tau) \in \mathbb{Z} \times \mathbb{R} \, | \, \xi \in I_k, \, |\tau - \omega(\xi)| \sim 2^j \}, \\
D_{k,\leq j} = \{(\xi,\tau) \in \mathbb{Z} \times \mathbb{R} \, | \, \xi \in I_k, \, |\tau - \omega(\xi)| \lesssim 2^j \}.
\end{split}
\end{equation*}

Recall the definition of the $X^{s,b}$-spaces (cf. \cite[Section~2.6]{Tao2006}) for a dispersion relation $\omega$, which were introduced in \cite{Bourgain1993FourierTransformRestrictionPhenomenaI,Bourgain1993FourierTransformRestrictionPhenomenaII}:
\begin{equation*}
\begin{split}
X^{s,b}_{\omega} &= \{ f \in \mathcal{S}^{\prime}(\R \times \T) \; | \; \Vert f \Vert_{X^{s,b}_{\omega}} < \infty \}, \\
\Vert f \Vert_{X^{s,b}_{\omega}} &= \Vert \langle n \rangle^{s} \langle \tau - \omega(n) \rangle^b \mathcal{F}_{t,x} f(\tau,n) \Vert_{L^2_\tau \ell^2_n}.
\end{split}
\end{equation*}
In the following we omit the subscript $\omega$ when we refer to the Airy dispersion relation $\omega(\xi)=\xi^3$.\\
We define an $X^{s,b}$-type space for the Fourier transform of frequency-localized functions:
\begin{equation*}
\begin{split}
&X_k = \{ f: \mathbb{R} \times \mathbb{Z} \rightarrow \mathbb{C} \; | \\
 &\mathrm{supp}(f) \subseteq \mathbb{R} \times {I_k}, \Vert f \Vert_{X_k} = \sum_{j=0}^\infty 2^{j/2} \Vert \eta_j(\tau - \omega(n)) f(\tau,n) \Vert_{\ell^2_{n} L^2_{\tau}} < \infty \} .
\end{split}
\end{equation*}

We recall the following estimates from \cite[p.~270,~Eqs.~(2.3),~(2.4)]{IonescuKenigTataru2008}:
\begin{equation*}
\Vert \int_{\mathbb{R}} | f_k(\tau^\prime,n) | d\tau^\prime \Vert_{\ell^2_{n}} \lesssim \Vert f_k \Vert_{X_k}
\end{equation*}
\begin{equation}
\label{eq:XkEstimateII}
\begin{split}
&\sum_{j=l+1}^{\infty} 2^{j/2} \Vert \eta_j(\tau - \omega(n)) \cdot \int_{\mathbb{R}} | f_k(\tau^\prime,n) | \cdot 2^{-l} (1+ 2^{-l}|\tau - \tau^\prime |)^{-4} d\tau^\prime \Vert_{L^2_\tau \ell^2_n} \\
&\quad + 2^{l/2} \Vert \eta_{\leq l} (\tau - \omega(n)) \cdot \int_{\mathbb{R}} | f_k(\tau^\prime,n) | \cdot 2^{-l} (1+ 2^{-l}|\tau - \tau^\prime |)^{-4} d\tau^\prime \Vert_{L^2_\tau \ell^2_{n}}\\
&\lesssim \Vert f_k \Vert_{X_k}
\end{split}
\end{equation}
\eqref{eq:XkEstimateII} implies for a Schwartz-function $\gamma$ and $k, l \in \mathbb{N}, t_0 \in \mathbb{R}, f_k \in X_k$ the estimate
\begin{equation}
\label{eq:XkEstimateIII}
\Vert \mathcal{F}_{t,x}[\gamma(2^l(t-t_0)) \cdot \mathcal{F}_{t,x}^{-1}(f_k)] \Vert_{X_k} \lesssim_{\gamma} \Vert f_k \Vert_{X_k}.
\end{equation}

We define 
\begin{equation*}
E_k = \{ u_0:\T \rightarrow \C \; | \; P_k u_0 = u_0, \; \Vert u_0 \Vert_{E_k} = \Vert u_0 \Vert_{L^2} < \infty \},
\end{equation*}
and we set
\begin{equation*}
C_0 (\mathbb{R}, E_k) = \left\{ u_k \in C(\mathbb{R},E_k) \; | \; \mathrm{supp}(u_k) \subseteq [-4,4] \times \mathbb{R} \right\}.
\end{equation*}

We define the short-time Fourier restriction space $F^{\alpha}_k$ for frequencies $2^{k}$ adapted to the time scale $2^{-k \alpha}$ by
\begin{equation*}
F^{\alpha}_k = \{ u_k \in C_0(\mathbb{R}, E_k) \; | \Vert u_k \Vert_{F^{\alpha}_k} = \sup_{t_k \in \mathbb{R}} \Vert \mathcal{F}_{t,x}[u_k \eta_0(2^{\alpha k}(t-t_k))] \Vert_{X_k} < \infty \}.
\end{equation*}
Similarly, we set for the space, in which the nonlinearity is estimated,
\begin{equation*}
\begin{split}
&N^{\alpha}_k = \{ u_k \in C_0(\mathbb{R}, E_k) \; | \\ 
&\quad \Vert u_k \Vert_{N^{\alpha}_k} = \sup_{t_k \in \mathbb{R}} \Vert (\tau - \omega(n) + i2^{\alpha k})^{-1} \mathcal{F}_{t,x}[u_k \eta_0(2^{\alpha k}(t-t_k))] \Vert_{X_k} < \infty \}.
\end{split}
\end{equation*}
The localization in time is carried out in a usual way. Set
\begin{equation*}
F^{\alpha}_k(T) = \{ u_k \in C([-T,T], E_k) \; | \Vert u_k \Vert_{F^{\alpha}_k(T)} = \inf_{\tilde{u}_k = u_k \mathrm{in} [-T,T] } \Vert \tilde{u}_k \Vert_{F^{\alpha}_k} < \infty \}
\end{equation*}
and
\begin{equation*}
N^{\alpha}_k(T) = \{ u_k \in C([-T,T], E_k) \; | \Vert u_k \Vert_{N^{\alpha}_k(T)} = \inf_{\tilde{u}_k = u_k \mathrm{in} [-T,T]} \Vert \tilde{u}_k \Vert_{N^{\alpha}_k} < \infty \}.
\end{equation*}

We assemble the spaces $E^s$, $E^s(T)$, $F^{s,\alpha}(T)$ and $N^{s,\alpha}(T)$ by means of Littlewood-Paley theory. The energy space for the initial data is given by
\begin{equation*}
E^{s} = \{ \phi: \mathbb{T} \rightarrow \mathbb{C} \; | \Vert \phi \Vert_{E^{s}}^2 = \sum_{k \geq 0} 2^{2 ks} \Vert P_k \phi \Vert_{L^2}^2 < \infty \}.
\end{equation*}
For the solution, we consider
\begin{equation*}
E^{s}(T) = \{ u \in C([-T,T],H^{\infty}) \; | \Vert u \Vert_{E^{s}(T)}^2 = \sum_{k \geq 0} \sup_{t_k \in [-T,T]} 2^{2 ks} \Vert P_k u(t_k) \Vert_{L^2}^2 < \infty \}.
\end{equation*}
We define the short-time Fourier restriction space for the solution by
\begin{equation*}
F^{s,\alpha}(T) = \{ u \in C([-T,T],H^{\infty}) \; | \Vert u \Vert_{F^{s,\alpha}(T)}^2 = \sum_{k \geq 0} 2^{2 ks} \Vert P_k u \Vert_{F^{\alpha}_k(T)}^2 < \infty \}.
\end{equation*}
For the nonlinearity, we consider
\begin{equation*}
N^{s,\alpha}(T) = \{ u \in C([-T,T],H^{\infty}) \; |  \Vert u \Vert_{N^{s,\alpha}(T)}^2 = \sum_{k \geq 0} 2^{2 ks} \Vert P_k u \Vert_{N^{\alpha}_k(T)}^2 < \infty \}.
\end{equation*}

Throughout this article, we will work with the renormalized version \eqref{eq:renormalizedmKdV} of \eqref{eq:mKdV}. We use the following notation for the trilinear interaction in \eqref{eq:renormalizedmKdV}:
\begin{equation*}
\begin{split}
\mathcal{F}_x \mathfrak{N}(u,v,w)(n) &= i n \hat{u}(n) \hat{v}(-n) \hat{w}(n) + in \sum_{\substack{n_1+n_2+n_3=n,\\(n_1+n_2)(n_1+n_3)(n_2+n_3)\neq 0}} \hat{u}(n_1) \hat{v}(n_2) \hat{w}(n_3) \\
&=: \mathcal{F}_x \mathcal{R}(u,v,w)(n) + \mathcal{F}_x \mathcal{N}(u,v,w)(n).
\end{split}
\end{equation*}
We abbreviate the condition $(n_1+n_2)(n_1+n_3)(n_2+n_3) \neq 0$ in the sum for the non-resonant interaction $\mathcal{N}$ with $(*)$, and in Fourier variables we write
\begin{equation*}
(f_1 * f_2 * f_3)^{\mathfrak{N}}(n) = \sum_{\substack{n_1+n_2+n_3 = n,\\(*)}} f_1(n_1) f_2(n_2) f_3(n_3), \quad f_i : \Z \rightarrow \C.
\end{equation*}

We turn to the basic properties of the function spaces introduced above. The following lemma deals with the embedding $F^{s,\alpha}(T) \hookrightarrow C([0,T],H^s)$.
\begin{lemma}
\label{lem:embeddingShorttimeSpaces}
\begin{enumerate}
\item[(i)] Let $u \in F^\alpha_k$. Then, we find the estimate
\begin{equation*}
\Vert u \Vert_{L_t^{\infty} L_x^2} \lesssim \Vert u \Vert_{F_k^\alpha}
\end{equation*}
to hold uniformly in $k$.
\item[(ii)] Let $s \in \R$ and $T > 0$ and $u \in F^{s,\alpha}(T)$. Then, we find the following estimate to hold:
\begin{equation*}
\Vert u \Vert_{C([0,T],H^s_x)} \lesssim \Vert u \Vert_{F^{s,\alpha}(T)}.
\end{equation*}
\end{enumerate}
\end{lemma}
\begin{proof}
See \cite[Lemma~3.1.,~p.~274]{IonescuKenigTataru2008} for the real line case and \cite[Lemma~3.2.,~p.~1668]{GuoOh2018} for the periodic case.
\end{proof}
For the large data theory, we have to define the following generalizations in terms of regularity $b \in \R$ in the modulation variable to the $X_k$-spaces
\begin{equation*}
\begin{split}
&X^b_k = \{ f: \mathbb{R} \times \mathbb{Z} \rightarrow \mathbb{C} \; | \\
 &\quad \mathrm{supp}(f) \subseteq \mathbb{R} \times {I_k}, \Vert f \Vert_{X^b_k} = \sum_{j=0}^\infty 2^{bj} \Vert \eta_j(\tau - \omega(n)) f(\tau,n) \Vert_{\ell^2_{n} L^2_{\tau}} < \infty \}.
\end{split}
\end{equation*}

The short-time spaces $F^{b,\alpha}_k$, $F^{b,s,\alpha}(T)$ and $N^{b,\alpha}_k$, $N^{b,s,\alpha}(T)$ are defined following along the above lines with $X_k$ replaced by $X^b_k$.\\
Indeed, in a similar spirit to the treatment of $X^{s,b}_T$-spaces, we can trade regularity in the modulation variable for a small power of $T$.
\begin{lemma}{\cite[Lemma~3.4.,~p.~1670]{GuoOh2018}}
\label{lem:tradingModulationRegularity}
Let $T>0$ and $b<1/2$. Then, for any function $u$ with temporal support in $[-T,T]$, we find the following estimate to hold:
\begin{equation*}
\Vert P_k u \Vert_{F_k^{b,\alpha}} \lesssim T^{1/2-b-} \Vert P_k u \Vert_{F_k^\alpha}.
\end{equation*}
\end{lemma}

Below, we have to consider the action of sharp time cutoffs in the $X_k$-spaces. Recall from the usual theory for Fourier restriction spaces that multiplication with a sharp cutoff in time is not bounded. However, we have the following lemma.
\begin{lemma}{\cite[Lemma~3.5.,~p.~1671]{GuoOh2018}}
\label{lem:sharpTimeCutoffAlmostBounded}
Let $k \in \Z$. Then, for any interval $I=[t_1,t_2] \subseteq \R$, we find the following estimate to hold:
\begin{equation*}
\sup_{j \geq 0} 2^{j/2} \Vert \eta_j(\tau-\omega(n)) \mathcal{F}_{t,x}[1_{I}(t) P_k u] \Vert_{L_\tau^2 \ell^2_n} \lesssim \Vert \mathcal{F}_{t,x} (P_k u) \Vert_{X_k}
\end{equation*}
with implicit constant independent of $k$ and $I$.
\end{lemma}
\section{Proof of Theorems \ref{thm:localExistence} and \ref{thm:nonExistence}}
\label{section:Proof}
The proof of Theorem \ref{thm:localExistence} is divided up into two parts: Firstly, we establish a priori estimates on smooth solutions. Next, a compactness argument is used to construct the solution mapping.
\begin{lemma}
\label{lem:aPrioriEstimateSmoothSolutions}
Let $u_0 \in H^{\infty}(\T)$ and $s>0$. There is a function $T=T(s,\Vert u_0 \Vert_{H^s})$ such that we find the following estimate for the unique smooth solution to \eqref{eq:renormalizedmKdV} to hold:
\begin{equation*}
\sup_{t \in [-T,T]} \Vert u(t) \Vert_{H^s(\T)} \lesssim \Vert u_0 \Vert_{H^s(\T)}.
\end{equation*}
\end{lemma}
For the proof, the $F^{s,1}(T)$-norm of the solution is bootstrapped. By virtue of Lemma \ref{lem:embeddingShorttimeSpaces}, this is enough to prove Lemma \ref{lem:aPrioriEstimateSmoothSolutions}.
Propagation of the $F^s$-norm is achieved by the following linear estimate:
\begin{lemma}
\label{lem:linearEstimatesShorttimeSpaces}
Let $\alpha >0$ and let $u$ be a smooth solution to the inhomogeneous equation
\begin{equation*}
\partial_t u + \partial_{xxx} u = v \; \mathrm{ on } \; [-T,T] \times \T
\end{equation*}
with $v \in C([-T,T],H^\infty(\T))$. Then, we find the following estimate to hold:
\begin{equation*}
\Vert u \Vert_{F^{s,\alpha}(T)} \lesssim \Vert u \Vert_{E^s(T)} + \Vert v \Vert_{N^{s,\alpha}(T)}.
\end{equation*}
\end{lemma}
\begin{proof}
The original proof from \cite{IonescuKenigTataru2008} for the real line carries over.
\end{proof}
Together with the nonlinear estimate from Proposition \ref{prop:nonlinearEstimate} and the energy estimate from Proposition \ref{prop:energyPropagation}, there is $\theta > 0$ and $c(s), d(s) > 0$ so that the following estimates hold true for any $M \in 2^{\N}$:
\begin{equation}
\label{eq:completePropagation}
\left\{\begin{array}{cl}
\Vert u \Vert_{F^{s,1}(T)} &\lesssim \Vert u \Vert_{E^{s}(T)} + \Vert \mathfrak{N}(u) \Vert_{N^{s,1}(T)} \\
\Vert \mathfrak{N}(u) \Vert_{N^{s,1}(T)} &\lesssim T^\theta \Vert u \Vert_{F^{s,1}(T)}^3 \\
\Vert u \Vert_{E^s(T)}^2 &\lesssim \Vert u_0 \Vert_{H^s}^2 + T^\theta M^{c(s)} \Vert u \Vert^4_{F^{s,1}(T)} \\
 &\quad + M^{-d(s)} \Vert u \Vert^4_{F^{s,1}(T)} + T^\theta \Vert u \Vert_{F^{s,1}(T)}^6
\end{array} \right.
\end{equation}
To carry out the continuity argument, we also need the limit properties of the involved norms:
\begin{lemma}
\label{lem:limitProperties}
Let $u,v \in C([-T,T],H^\infty(\T))$. We find the mappings $T \mapsto \Vert v \Vert_{N^s(T^)} $, $T \mapsto \Vert u \Vert_{E^s(T)}$ to be continuous, and we have
\begin{align*}
\lim_{T \rightarrow 0} \Vert v \Vert_{N^{s,\alpha}(T)} &= 0, \\
\lim_{T \rightarrow 0} \Vert u \Vert_{E^s(T)} &= \Vert u_0 \Vert_{H^s}.
\end{align*}
\end{lemma}
\begin{proof}
See \cite[Lemma~4.2.,~Eq.~(4.6),~p.~279]{IonescuKenigTataru2008} for the original proof on the real line and \cite[Lemma~8.1.,~p.~1719]{GuoOh2018} for a proof of these properties on the torus.
\end{proof}
We are ready to prove a priori estimates for smooth solutions. The argument below is standard when establishing a priori estimates in the context of short-time Fourier restriction spaces (cf. \cite[Proposition~4.1.,~p.~270]{IonescuKenigTataru2008}).
\begin{proof}[Proof of Lemma \ref{lem:aPrioriEstimateSmoothSolutions}]
Assuming that $u_0$ is a smooth and real-valued initial datum, we find from the classical well-posedness theory the global existence of a smooth and real-valued solution $u \in C(\mathbb{R},H^{\infty})$ (see e.g. \cite{Bourgain1993FourierTransformRestrictionPhenomenaI}), which satisfies the set of estimates \eqref{eq:completePropagation}.

We define $X(T)= \Vert u \Vert_{E^{s}(T)} + \Vert \mathfrak{N}(u) \Vert_{N^{s,1}(T)}$ and find the bound 
\begin{equation*}
 X(T)^2 \leq C_1 \Vert u_0 \Vert_{H^s}^2 + C_2((T^\theta M^{c(s)} + M^{-d(s)})X(T)^2 + T^\theta X(T)^4) X(T)^2
\end{equation*} 
by eliminating $\Vert u \Vert_{F^{s,\alpha}(T)}$ in the system of estimates \eqref{eq:completePropagation}.\\
Set $R=C_1^{1/2} \Vert u_0 \Vert_{H^s}$ and choose $M=M(R)$ large enough so that $C_2 M^{-d(s)} (2R)^2 < 1/4$.

Next, choose $T_0 = T_0(R) \leq 1$ small enough so that $C_2 T_0^\theta (M^{c(s)}(2R)^2 + (2R)^4) < 1/4$. Together with Lemma \ref{lem:limitProperties}, a continuity argument yields $X(T) \leq 2R$ for $T \leq T_0$. Iterating the argument yields $\sup_{t \in [0,T_0]} \Vert u(t) \Vert_{H^s(\T)} \lesssim \Vert u_0 \Vert_{H^s}$ for $T_0 = T_0(\Vert u_0 \Vert_{H^s})$. The proof is complete.
\end{proof}
We turn to establishing the existence of the solution mapping. For $u_0 \in H^s(\T)$, we set $u_{0,n} = P_{\leq n} u_0$ for $n \in \mathbb{N}$. Obviously, $u_{0,n} \in H^{\infty}(\T)$, hence the initial data give rise to smooth global solutions $u_n \in C(\R,H^{\infty}(\T))$. According to Lemma \ref{lem:aPrioriEstimateSmoothSolutions}, we have a priori estimates on a time interval $[0,T_0]$ where $T_0 = T_0(\Vert u_0 \Vert_{H^s})$ independent of $n$. Moreover, we have the following compactness lemma. In the context of short-time $X^{s,b}$-spaces the below arguments were given in \cite[Lemma~8.2.,~p.~1724]{GuoOh2018} for the Wick-ordered cubic NLS. Thus, the proof is omitted.
\begin{lemma}
\label{lem:compactnessLemma}
Let $u_0 \in H^s(\T)$ for some $s>0$. Let $u_n$ be the smooth global solutions to \eqref{eq:renormalizedmKdV} with $u_n(0) = u_{0,n}$ like above.

Then, $(u_n)_{n \in \N}$ is precompact in $C([-T,T],H^s(\T))$ for $T \leq T_0 = T_0(\Vert u_0 \Vert_{H^s})$.
\end{lemma}
We are ready to prove the main result:
\begin{proof}[Proof of Theorem \ref{thm:localExistence}]
For $u_0 \in H^s(\T)$ let $(u_n)_{n \in \mathbb{N}}$ be the smooth global solutions generated from the initial data $P_{\leq n} u_0$ like described above. By Lemma \ref{lem:compactnessLemma}, we find a convergent subsequence $(u_{n_k})$ which converges to a function $u \in C([-T,T],H^s)$. Due to a uniform tail estimate from the proof of Lemma \ref{lem:compactnessLemma}, the sequence also converges in $E^s(T)$. With $\Vert \mathfrak{N}(u_n - u) \Vert_{N^{s,1}(T)} \lesssim T^\theta \Vert u_0 \Vert_{H^s}^2 \Vert u_n - u \Vert_{F^{s,1}(T)}$, we find for $T= T(\Vert u_0 \Vert_{H^s})$ the estimate
\begin{equation*}
\Vert u_n - u \Vert_{F^{s,1}(T)} \lesssim \Vert u_n - u \Vert_{E^s(T)}
\end{equation*}
to hold. The convergence in $F^{s,1}(T)$ already gives the a priori estimate for the limit. Moreover, we deduce from the multilinear estimates in Proposition \ref{prop:nonlinearEstimate} that $( \mathfrak{N}(u_n) )$ converges to $\mathfrak{N}(u)$ in $N^{s,1}(T) \hookrightarrow \mathcal{D}^\prime$. We conclude that $u$ satisfies \eqref{eq:renormalizedmKdV} in the sense of generalized functions with the claimed properties, and the proof is complete. 
\end{proof}
For the proof of Theorem \ref{thm:nonExistence} one compares smooth solutions to \eqref{eq:mKdV} and \eqref{eq:renormalizedmKdV} via a gauge transform. Invoking the Riemann-Lebesgue lemma, the existence of solutions to \eqref{eq:renormalizedmKdV} rules out the existence of non-trivial solutions to \eqref{eq:mKdV}. For details we refer to \cite[Section~9]{GuoOh2018}.
\section{Multilinear estimates}
\label{section:MultilinearEstimates}
In the following we recall and derive multilinear estimates for functions with support of the space-time Fourier transform adapted to the Airy equation. We denote the frequency ranges by $k_i$ and the modulation ranges by $j_i$. The decreasing arrangements are denoted by $k_i^*$ or $j_i^*$, respectively.

We recall the following linear Strichartz estimates going back to Bourgain (cf. \cite{Bourgain1993FourierTransformRestrictionPhenomenaI,Bourgain1993FourierTransformRestrictionPhenomenaII}):
\begin{lemma}
\label{lem:linearStrichartzEstimates}
Given $u \in X^{0,1/3}$, the estimate
\begin{equation}
\label{eq:L4StrichartzEstimate}
\Vert u \Vert_{L^4_{t}(\R,L_x^4(\T))} \lesssim \Vert u \Vert_{X^{0,1/3}}
\end{equation}
holds true.

Given $u_0 \in L^2(\T)$ with $\mathrm{supp}(\hat{u}_0) \subseteq [-N,N]$, we find
\begin{equation}
\label{eq:L6StrichartzEstimate}
\Vert S(t) u_0 \Vert_{L^6_{t,x}(\T^2)} \lesssim C_\varepsilon N^\varepsilon \Vert u_0 \Vert_{L^2(\T)}.
\end{equation}
\end{lemma}
\begin{proof}
Estimate \eqref{eq:L4StrichartzEstimate} is proved in \cite[Proposition~7.15.,~p.~211]{Bourgain1993FourierTransformRestrictionPhenomenaII}, and \eqref{eq:L6StrichartzEstimate} is \cite[Estimate~(8.37),~p.~227]{Bourgain1993FourierTransformRestrictionPhenomenaII}. 
\end{proof}
By the above estimates, we find the following due to H\"older's inequality and almost orthogonality:
\begin{lemma}
\label{lem:multilinearStrichartzEstimates}
For $u \in L^2(\R \times \T)$ with $\mathrm{supp}(\tilde{u}_i) \subseteq D_{k_i,\leq j_i}$ we find the following estimates to hold:
\begin{align}
\label{eq:multilinearL4StrichartzEstimate}
\int_{\R \times \T} u_1 u_2 u_3 u_4 dt dx &\lesssim \prod_{i=1}^4 2^{j_i/3} \Vert \mathcal{F}_{t,x} (u_i) \Vert_{L_{\tau}^2 \ell_n^2}, \\
\label{eq:multilinearL6StrichartzEstimate}
\int_{\R \times \T} u_1 u_2 u_3 u_4 dx dt &\lesssim 2^{-j_1^*/2} 2^{\varepsilon k_3^*} \prod_{i=1}^4 2^{j_i/2} \Vert \mathcal{F}_{t,x}(u_i) \Vert_{L_{\tau}^2 \ell_n^2}.
\end{align}
\end{lemma}
\begin{proof}
Estimate \eqref{eq:multilinearL4StrichartzEstimate} follows from an application of H\"older's inequality. For a proof of \eqref{eq:multilinearL6StrichartzEstimate}, see for instance \cite[Equation~(5.5),~p.~1682]{GuoOh2018}. 
\end{proof}
In \cite{Bourgain1993FourierTransformRestrictionPhenomenaII} was conjectured that the estimate
\begin{equation*}
\Vert u \Vert_{L^8_{t,x}(\mathbb{R} \times \T)} \lesssim \Vert u \Vert_{X^{0+,1/2+}}
\end{equation*}
holds true. Interpolation with \eqref{eq:L4StrichartzEstimate} gives
\begin{equation*}
\Vert u \Vert_{L^6_{t,x}(\mathbb{R} \times \T)} \lesssim \Vert u \Vert_{X^{0+,4/9+}}.
\end{equation*}
This estimate would provide us with smoothing in any short-time $F^\alpha$-space and seems to be necessary to carry out energy estimates in negative Sobolev spaces. Dinh showed the short-time estimate \cite[Proposition~2.5,~p.~8812]{Dinh2017}
\begin{equation}
\label{eq:ShorttimeL6Estimate}
\Vert P_n e^{t \partial_x^3} u_0 \Vert_{L_t^6([0,2^{-2n}],L^6(\T))} \lesssim 2^{-n/6} \Vert P_n u_0 \Vert_{L^2}.
\end{equation}
We infer that the $L^6_{t,x}$-Strichartz estimate loses no derivatives in the $F^1$-space by the following
\begin{equation*}
\begin{split}
\Vert P_n e^{t \partial_x^3} u_0 \Vert_{L_t^6([0,2^{-n}],L^6(\T))} &\lesssim \left( \sum_{\substack{I \subseteq [0,2^{-n}], \\ I: |I| = 2^{-2n}}} \Vert P_n e^{t \partial_x^3} u_0 \Vert^6_{L_t^6(I,L^6(\T))} \right)^{1/6} \\
&\lesssim \Vert P_n u_0 \Vert_{L^2}
\end{split}
\end{equation*}
The smoothing obtained in the $F^{\alpha}$-spaces for $\alpha > 1$ by \eqref{eq:ShorttimeL6Estimate} is insufficient to prove energy estimates in negative Sobolev spaces.

We recall the following bilinear estimate from \cite{Molinet2012}.
\begin{lemma}[{\cite[Equation~(3.7),~p.~1906]{Molinet2012}}]
\label{lem:bilinearEstimate}
Let $f_1, f_2 \in L^2(\R \times \Z)$ with the following support properties
\begin{equation*}
(\tau,n) \in \; \text{supp} \; (f_i) \Rightarrow \langle \tau - n^3 \rangle \lesssim 2^{j_i}, \quad i=1,2,
\end{equation*}
where $j_1 \leq j_2$.\\
Then, for any $2^k>0$, we find the following estimate to hold:
\begin{equation}
\label{eq:BilinearEstimate}
\Vert f_1 * f_2 \Vert_{L_\tau^2 \ell_n^2(|n| \geq 2^k)} \lesssim 2^{j_1/2} \left( 2^{(j_2-k)/4} + 1 \right) \Vert f_1 \Vert_{L^2} \Vert f_2 \Vert_{L^2}.
\end{equation}
\end{lemma}
In case of separated frequencies, we can refine the above estimates. 
The following lemma is adapted to the nonlinear interaction dictated by the modified Korteweg-de Vries equation.
If there is one frequency significantly lower than the remaining three, the resonance is very favourable, and we do not need a refined estimate. Thus, we only consider the case where two frequencies are smaller than the remaining two, which is relevant for $High \times Low \times Low \rightarrow High$-interaction:
 \begin{lemma}
 \label{lem:RefinedHighLowLowHighMultilinearEstimate}
 Suppose that $k_4 \geq 20$, $k_1 \leq k_2 \leq k_3-5$ and $u_i \in L^2(\mathbb{R} \times \mathbb{Z})$ for $i=1,\ldots,4$. Moreover, suppose that $j_i \geq [\alpha k_1^*]$ for $i =1, \ldots, 4$ with $\alpha \leq 2$ and $\text{supp}(\tilde{u}_i) \subseteq D_{k_i,\leq j_i}$, and suppose that $supp_{n} (\tilde{u}_i) \subseteq J_i$, where $|J_i| \lesssim 2^l$.\\
 Then, we find the following estimate to hold: 
 \begin{equation}
 \label{eq:HighLowLowHighInteractionEstimate}
 \int_{\R \times \T} \; u_1(t,x) u_2(t,x) u_3(t,x) u_4(t,x) dt dx \lesssim M \prod_{i=1}^4 2^{j_i/2} \Vert \mathcal{F}_{t,x} (u_i) \Vert_{L^2_\tau \ell^2_n},
 \end{equation}
 where $M = 2^{l/2} 2^{-j_1^*/2} 2^{-[\alpha k_1^*]/2}$.
 \end{lemma}
\begin{proof}
We denote the space-time Fourier transform of $u_i:\R \times \T \rightarrow \C$ by $f_i:\R \times \Z \rightarrow \C$, $\tilde{u}_i(\tau,n) = f_i(\tau,n)$.\\
Further, we consider the shifted function $g_i(\tau,n) = f_i(\tau+n^3,n)$ and observe $(\tau,n) \in \mathrm{supp}(g_i) \Leftrightarrow (\tau+n^3,n) \in \mathrm{supp}(f_i)$. Hence, for $(\tau,n) \in \text{supp} (g_i)$ we find $|\tau| \lesssim 2^{j_i}$.

 \textbf{Case A:} Suppose that $j_1^* = j_1$. That means a low frequency carries a high modulation. It will be easy to see that the computation below can also deal with the case $j_1^* = j_2$ by exchanging the roles of $g_1$ and $g_2$.\\
We find after a change of variables in Fourier space
\begin{equation}
\label{eq:QuadlinearExpressionFourierSpace}
\begin{split}
&\quad \int_{\T} dx \int_{\R} dt u_1(t,x) u_2(t,x) u_3(t,x) u_4(t,x) \\
&= \int_{\tau_1+\ldots+\tau_4=0} \sum_{n_1+\ldots+n_4=0} f_1(\tau_1,n_1) f_2(\tau_2,n_2) f_3(\tau_3,n_3) f_4(\tau_4,n_4) \\
&= \int_{\tau_1,\tau_3,\tau_4} \sum_{n_1,n_3,n_4} g_1(\tau_1,n_1) g_2(h(\tau_1,\tau_3,\tau_4,n_1,n_3,n_4),-n_1-n_3-n_4) \\
&\quad \quad g_3(\tau_3,n_3) g_4(\tau_4,n_4).
\end{split}
\end{equation} 
By means of the resonance function
\begin{equation*}
h(\tau_1,\tau_2,\tau_3,n_1,n_2,n_3)=-\tau_1-\tau_2-\tau_3+3(n_1+n_2)(n_1+n_3)(n_2+n_3),
\end{equation*}
we can compute the effective supports in the modulation variables.\\
Set
\begin{equation*}
E_{24} = \{ n_4 \in \mathbb{Z} | \, |h(\tau_1,\tau_3,\tau_4,n_1,n_3,n_4)| \lesssim 2^{j_2} \}.
\end{equation*}
Since the second variable is distinguished, we denote $h(\tau_1,\tau_3,\tau_4,n_1,n_3,n_4)$ by $h_2$ and compute $\partial_{n_4} h_2 = C (n_1+n_3)(n_4-n_2)$, which gives $|\partial_{n_4} h_2| \gtrsim 2^{2k_1^*}$.\\
Thus, an application of the Cauchy-Schwarz inequality yields $|E_{24}| \lesssim 1+2^{j_2-2k_1^*}$, and we derive
\begin{equation}
\label{eq:HighLowLowHighIntermediateMultilinearEstimateI}
 \begin{split}
 &\quad \sum_{n_1,n_3} \int d\tau_1 d\tau_3 d\tau_4 |g_1(\tau_1,n_1)| |g_3(\tau_3,n_3)| \sum_{n_4} |g_2(h_2,-n_1-n_3-n_4)| |g_4(\tau_4,n_4)| \\
 &\lesssim (1+2^{j_2-2k_1^*})^{1/2} \sum_{n_1,n_3} \int d\tau_3 |g_3(\tau_3,n_3)| \int d\tau_1 \int d\tau_4 |g_1(\tau_1,n_1)| \times\\
 &\quad \left( \sum_{n_4} |g_2(h_2,-n_1-n_3-n_4)|^2 |g_4(\tau_4,n_4)|^2 \right)^{1/2}.
\end{split} 
 \end{equation}
By repeated applications of the Cauchy-Schwarz inequality, it follows
 \begin{equation*}
 \begin{split}
 \eqref{eq:HighLowLowHighIntermediateMultilinearEstimateI} &\lesssim (1+2^{j_2-2k_1^*})^{1/2} \sum_{n_1,n_3} \int d\tau_3 |g_3(\tau_3,n_3)| \int d\tau_4 \left( \int d\tau_1 |g_1(\tau_1,n_1)|^2 \right)^{1/2} \\
 &\quad \left( \sum_{n_4} \Vert g_2(h_2,-n_1-n_3-n_4) \Vert_{L^2_{\tau_1}}^2 |g_4(\tau_4,n_4)|^2 \right)^{1/2} \\
 &\lesssim (1+2^{j_2-2k_1^*})^{1/2} \sum_{n_3} \int d\tau_3 |g_3(\tau_3,n_3)| \Vert g_1 \Vert_{L^2_{\tau} \ell^2_n} \Vert g_2 \Vert_{L^2_{\tau} \ell^2_n} \\
 &\quad \int d\tau_4 \left( \sum_{n_4} | g_4(\tau_4,n_4) |^2 \right)^{1/2} \\
 &\lesssim (1+2^{j_2-2k_1^*})^{1/2} 2^{l/2} 2^{j_3/2} 2^{j_4/2} \prod_{i=1}^4 \Vert g_i \Vert_{L^2_{\tau} \ell^2_n}.
 \end{split} 
 \end{equation*}
In case $j_2 \geq 2k_1^*$ we find \eqref{eq:HighLowLowHighInteractionEstimate} to hold with $M = 2^{l/2} 2^{-j_1^*/2} 2^{-k_1^*}$. If $j_2 \leq 2k_1^*$, we find \eqref{eq:HighLowLowHighInteractionEstimate} to hold with $M = 2^{l/2} 2^{-j_1^*/2} 2^{-[\alpha k_1^*]/2}$, which is the larger bound. This proves \eqref{eq:HighLowLowHighInteractionEstimate} in Case A.

 \textbf{Case B:} In case $j_1^* = j_3$, that is a high frequency carrying a high modulation, we use H\"older's inequality to find
 \begin{equation}
 \label{eq:HighLowLowHighIntermediateMultilinearEstimateII}
 \begin{split}
|\eqref{eq:QuadlinearExpressionFourierSpace}| &\lesssim \Vert g_2 \Vert_{L_{\tau}^2 \ell_n^2} 2^{j_2/2} 2^{l/2} \sup_{n_2,\tau_2} \int d\tau_1 \int d\tau_3 \sum_{n_3} |g_3(\tau_3,n_3)| \\
&\quad \sum_{n_1} |g_1(\tau_1,n_1)| |g_4(h_4,-n_1-n_2-n_3)|.
\end{split}
\end{equation}
We consider the set $E_{41} = \{ n_1 \in \Z \, | \, |h_4| \lesssim 2^{j_4} \}$. Since $\partial_{n_1} h_4 = 3(n_2+n_3)(n_2-n_4)$, we find $|\partial_{n_1} h_4| \gtrsim 2^{2k_1^*}$ and further $|E_{41}| \lesssim (1+ 2^{j_4-2k_1^*})^{1/2}$.

By repeated use of the Cauchy-Schwarz inequality, it follows
\begin{equation*}
\begin{split}
 \eqref{eq:HighLowLowHighIntermediateMultilinearEstimateII} &\lesssim 2^{j_2/2} 2^{l/2} \Vert g_2 \Vert_{L^2_{\tau} \ell^2_n} \sup_{\tau_2,n_2} \int d\tau_1 \int d\tau_3 \sum_{n_3} |g_3(\tau_3,n_3)| \\
 &\quad \quad \left( \sum_{n_1} |g_1(\tau_1,n_1)|^2 |g_4(h_4,-n_1-n_2-n_3)|^2 \right)^{1/2} \\
 &\lesssim 2^{j_2/2} 2^{l/2} (1+2^{j_4-2k_1^*})^{1/2} \Vert g_2 \Vert_{L^2_{\tau} \ell^2_n} \sup_{n_2,\tau_2} \int d\tau_1 \sum_{n_3} \left( \int d\tau_3 |g_3(\tau_3,n_3)|^2 \right)^{1/2} \\
 &\quad \quad \left( \sum_{n_1} |g_1(\tau_1,n_1)|^2 \Vert g_4(h_4,-n_1-n_2-n_3) \Vert_{L^2_{\tau_3}}^2 \right)^{1/2} \\
 &\lesssim 2^{j_2/2} 2^{l/2} (1+2^{j_4-2k_1^*})^{1/2} \Vert g_2 \Vert_{L^2_{\tau} \ell^2_{n}} \Vert g_3 \Vert_{L^2_{\tau} \ell^2_n} \Vert g_4 \Vert_{L^2_{\tau} \ell^2_n} \\
 &\quad \quad \int d\tau_1 \left( \sum_{n_1} |g_1(\tau_1,n_1)|^2 \right)^{1/2} \\
 &\lesssim 2^{j_1/2} 2^{j_2/2} 2^{l/2} (1+2^{j_4-2k_1^*})^{1/2} \prod_{i=1}^4 \Vert g_i \Vert_{L^2_{\tau} \ell^2_n}.
 \end{split}
 \end{equation*}
 The estimate \eqref{eq:HighLowLowHighInteractionEstimate} follows from the same considerations as in Case A.

  Clearly, an adapted computation shows the claim if $j_1^* = j_4$. The proof is complete. 
\end{proof}
The estimate for $High \times High \times Low \rightarrow Low$-interaction is related, but the minimal size of the support of the modulation variable is different.
\section{Short-time trilinear estimates}
\label{section:shorttimeTrilinearEstimates}
Our aim is to prove estimates of the following kind for all possible frequency interactions:
\begin{equation}
\label{eq:shorttimeFrequencyLocalizedEstimate}
\Vert P_{k_4} \mathcal{N}(u_1,u_2,u_3) \Vert_{N_{k_4}^\alpha} \lesssim \underbrace{D(\alpha,k_1,k_2,k_3,k_4)}_{D(\alpha,\underline{k})} \Vert u_1 \Vert_{F_{k_1}^{1/2-,\alpha}} \Vert u_2 \Vert_{F_{k_2}^{1/2-,\alpha}} \Vert u_3 \Vert_{F_{k_3}^{1/2-,\alpha}}
\end{equation}
In fact, the resonant interaction can be perceived as a special case of $High \times High \times High \rightarrow High$-interaction, see below. Hence, we only estimate the non-resonant part.

The trilinear estimate
\begin{equation}
\label{eq:shorttimeTrilinearEstimate}
\Vert \mathfrak{N}(u_1,u_2,u_3) \Vert_{N^{s,\alpha}(T)} \lesssim T^\theta \Vert u_1 \Vert_{F^{s,\alpha}(T)} \Vert u_2 \Vert_{F^{s,\alpha}(T)} \Vert u_3 \Vert_{F^{s,\alpha}(T)}
\end{equation}
then follows from splitting up the frequency support of the functions and Lemma \ref{lem:tradingModulationRegularity}. Note that it will be enough to estimate one function in \eqref{eq:shorttimeFrequencyLocalizedEstimate} with a modulation regularity slightly below $1/2$ to derive \eqref{eq:shorttimeTrilinearEstimate}.

Below, we only prove \eqref{eq:shorttimeFrequencyLocalizedEstimate} for $F_{k_i}^\alpha$-spaces in detail.
 The systematic modification to find \eqref{eq:shorttimeFrequencyLocalizedEstimate} to hold with one modulation regularity strictly less than $1/2$ follows from accepting a slight loss in the highest modulation.
 
We start with $High \times Low \times Low \rightarrow High$-interaction.
\begin{lemma}
\label{lem:ShorttimeHighLowLowHighInteraction}
Let $k_4 \geq 20$, $k_1 \leq k_2 \leq k_3-5$ and suppose that $P_{k_i} u_i = u_i$ for $i \in \{1,2,3\}$. Then, we find the estimate \eqref{eq:shorttimeFrequencyLocalizedEstimate} to hold with $D(\alpha,\underline{k}) = 2^{-(\alpha/2-\varepsilon)k_4}$ for any $\varepsilon > 0$.
\end{lemma}
\begin{proof}
Let $\gamma: \R \rightarrow [0,1]$ be a smooth function with $\text{supp}(\gamma) \subseteq [-1,1]$ and 
\begin{equation*}
\sum_{m \in \Z} \gamma^3(x-m) \equiv 1.
\end{equation*}
We find the left-hand side in \eqref{eq:shorttimeFrequencyLocalizedEstimate} to be dominated by 
\begin{equation*}
\begin{split}
&C 2^{k_4} \sum_{m \in \Z} \sup_{t_{k_4} \in \R} \Vert (\tau - n^3 +i2^{[\alpha k_4]})^{-1} 1_{I_{k_4}}(n) \\
&\quad ( \mathcal{F}_{t,x}[\eta_0(2^{[\alpha k_4]}(t-t_{k_4})) \gamma(2^{[\alpha k_1^*]}(t-t_{k_4})-m) u_1] \\
&\quad \quad * \mathcal{F}_{t,x}[\gamma(2^{[\alpha k_1^*]}(t-t_{k_4})-m) u_2] * \mathcal{F}_{t,x}[\gamma(2^{[\alpha k_1^*]}(t-t_{k_4})-m) u_3])^{\mathfrak{N}} \Vert_{X_{k_4}}.
\end{split}
\end{equation*}

We observe that $\# \{ m \in \Z | \eta_0(2^{[\alpha k_4]}(\cdot -t_k)) \gamma(2^{[\alpha k_1^*]}(\cdot -t_k)-m) \neq 0 \} = \mathcal{O}(1)$. Consequently, it is enough to prove
\begin{equation*}
\begin{split}
&C 2^{k_4} \sup_{t_{k_4} \in \R} \Vert (\tau - n^3 +i2^{[\alpha k_4]})^{-1} 1_{I_{k_4}}(n) ( \mathcal{F}_{t,x}[\eta_0(\ldots) \gamma(2^{[\alpha k_1^*]}(t-t_k)-m) u_1] \\
&\quad \quad * \mathcal{F}_{t,x}[\gamma(2^{[\alpha k_1^*]}(t-t_k)-m) u_2] * \mathcal{F}_{t,x}[\gamma(2^{[\alpha k_1^*]}(t-t_k)) u_3])^{\mathfrak{N}} \Vert_{X_{k_4}} \\
&\lesssim_{\varepsilon} 2^{-(\alpha/2-\varepsilon)k_4} \Vert u_1 \Vert_{F_{k_1}^\alpha} \Vert u_2 \Vert_{F_{k_2}^\alpha} \Vert u_3 \Vert_{F_{k_3}^\alpha}.
\end{split}
\end{equation*}

We write $f_{k_i} = \mathcal{F}_{t,x}[\eta_0(2^{[\alpha k_4]}(t-t_k) \gamma(2^{[\alpha k_1^*]}(t-t_k)-m) u_i]$, and to denote additional localization in modulation, we use the notation
\begin{equation*}
f_{k_i,j_i} = 
\begin{cases}
	\eta_{\leq j_i}(\tau - n^3) f_{k_i}, \; j_i = [\alpha k_1^*], \\
	\eta_{j_i}(\tau - n^3) f_{k_i}, \; j_i > [\alpha k_1^*].
\end{cases}
\end{equation*}
By means of the definition of $F_{k_i}^\alpha$ and \eqref{eq:XkEstimateIII}, it is further enough to prove 
\begin{equation}
\label{eq:ReductionShorttimeHighLowLowHighEstimate}
\begin{split}
&\quad \sum_{j_4 \geq [\alpha k_4]} \sum_{j_1,j_2,j_3 \geq [\alpha k_1^*]} 2^{-j_4/2} \Vert 1_{D_{k_4,\leq j_4}} (f_{k_1,j_1} * f_{k_2,j_2} * f_{k_3,j_3})^{\mathfrak{N}} \Vert_{L^2_\tau \ell^2_n} \\
&\lesssim_{\varepsilon} 2^{-( \alpha/2 - \varepsilon ) k_1} \prod_{i=1}^3 \sum_{j_i \geq [\alpha k_	1^*]} 2^{j_i/2} \Vert f_{k_i,j_i} \Vert_{L^2_{\tau} \ell^2_n}.
\end{split}
\end{equation}

We see that \eqref{eq:ReductionShorttimeHighLowLowHighEstimate} follows from \eqref{eq:HighLowLowHighInteractionEstimate}. The resonance function, yielding a lower bound for $j_1^*$ in \eqref{eq:ReductionShorttimeHighLowLowHighEstimate}, is given by
\begin{equation*}
\Omega=(k_1+k_2+k_3)^3 - k_1^3 - k_2^3 - k_3^3 = 3(k_1+k_2)(k_1+k_3)(k_2+k_3).
\end{equation*}
Thus, $2^{2k_1^*} \lesssim |\Omega| \lesssim 2^{2k_1^*+k_3^*}$. To derive effective estimates, we localize $|\Omega| \sim 2^{2k_1+l}$. This is equivalent to prescribing $|k_1+k_2| \sim 2^l$, and the contribution to \eqref{eq:ReductionShorttimeHighLowLowHighEstimate} is denoted by
\begin{equation*}
\Vert P^l_{\Omega} 1_{D_{k_4,j_4}} (f_{k_1,j_1} * f_{k_2,j_2} * f_{k_3,j_3})^{\mathfrak{N}} \Vert_{L^2_{\tau} \ell^2_n}.
\end{equation*}

In the above display, we split the frequency support of $f_{k_1,j_1}$ into intervals of length $2^l$, that is $f_{k_1,j_1} = \sum_{I_1} f_{k_1,j_1}^{I_1}$. Due to localization of $\Omega$, this also gives a decomposition of $f_{k_2,j_2}$ so that the above display is dominated by
\begin{equation*}
\sum_{I_1,I_2} \Vert P^l_{\Omega} 1_{D_{k_4,j_4}} (f^{I_1}_{k_1,j_1} * f^{I_2}_{k_2,j_2} * f_{k_3,j_3})^{\mathfrak{N}} \Vert_{L^2_{\tau} \ell^2_n}.
\end{equation*}
Further, we split after decomposition in $0 \leq l \leq k_3^*$ the sum over $j_4$ into $j_4 \leq 2 k_1^* + l$ and $j_4 \geq 2k_1^* + l$. For fixed $l$, we find from \eqref{eq:HighLowLowHighInteractionEstimate}
\begin{equation*}
\begin{split}
&\quad 2^{k_4} \sum_{[\alpha k_4] \leq j_4 \leq 2k_1^* + l} \sum_{\substack{j_i \geq [\alpha k_1^*], \\ i =1,2,3}} 2^{-j_4/2} \sum_{I_1,I_2} \Vert P^l_{\Omega} 1_{D_{k_4,\leq j_4}} (f^{I_1}_{k_1,j_1} * f^{I_2}_{k_2,j_2} * f_{k_3,j_3})^{\mathfrak{N}} \Vert_{L^2_{\tau} \ell^2_n} \\
&\lesssim 2^{k_4} \sum_{[\alpha k_4] \leq j_4 \leq 2k_1^* + l} 2^{-j_4/2} \sum_{\substack{j_i \geq [\alpha k_1^*], \\ i =1,2,3}} 2^{-j_1^*/2} 2^{l/2} 2^{-[\alpha k_1]/2} 2^{j_4/2} \prod_{i=1}^3 2^{j_i/2} \Vert f_{k_i,j_i} \Vert_{L^2} \\
&\lesssim k_1^* 2^{-[\alpha k_1]/2} \prod_{i=1}^3 \sum_{j_i \geq [\alpha k_1^*]} 2^{j_i/2} \Vert f_{k_i,j_i} \Vert_{L^2_{\tau} \ell^2_n},
\end{split}
\end{equation*}
where $f^{I_i}_{k_i,j_i}$ for $i=1,2$ were reassembled to $f_{k_i,j_i}$ by Cauchy-Schwarz inequality.\\
For the second part $j_4 \geq 2k_1^* + l$, we just take $2^{-j_1^*/2} \leq 2^{-j_4/2}$ to find in a similar spirit
\begin{equation*}
\begin{split}
&\quad 2^{k_4} \sum_{j_4 \geq 2k_1^* + l} \sum_{\substack{j_i \geq [\alpha k_1^*], \\ i =1,2,3}} 2^{-j_4/2} \sum_{I_1,I_2} \Vert P^l_{\Omega} 1_{D_{k_4,\leq j_4}}
(f^{I_1}_{k_1,j_1} * f^{I_2}_{k_2,j_2} * f_{k_3,j_3})^{\mathfrak{N}} \Vert_{L^2_{\tau} \ell^2_n} \\
&\lesssim 2^{k_4} \sum_{j_4 \geq 2k_1^* + l} 2^{-j_4/2} 2^{l/2} 2^{-[\alpha k_1]/2} \prod_{i=1}^3 2^{j_i/2} \Vert f_{k_i,j_i} \Vert_{L^2} \\
&\lesssim 2^{-[\alpha k_1]/2} \prod_{i=1}^3 \sum_{j_i \geq [\alpha k_1^*]} 2^{j_i/2} \Vert f_{k_i,j_i} \Vert_{L^2_{\tau} \ell^2_n}.
\end{split}
\end{equation*}

An estimate with one modulation size strictly less than $1/2$ follows from slight loss in the highest modulation. We omit the details. The proof is complete. 
\end{proof}
We turn to $High \times High \times Low \rightarrow High$-interaction.
\begin{lemma}
\label{lem:ShorttimeHighHighLowHighInteraction}
Let $k_4 \geq 20$, $k_1 \leq k_2 \leq k_3$, $k_1 \leq k_2-15$ and $|k_2-k_4| \leq 10$ and suppose that $P_{k_i} u_i = u_i$ for $i \in \{1,2,3\}$. Then, we find estimate \eqref{eq:shorttimeFrequencyLocalizedEstimate} to hold with $D(\alpha,\underline{k}) = 2^{-(1/2-\varepsilon) k_4}$ for any $\varepsilon >0$.
\end{lemma}
\begin{proof}
By the reductions and notation from above, we have to prove
\begin{equation}
\label{eq:ReductionShorttimeHighHighLowHighEstimate}
\begin{split}
&\quad 2^{k_4} \sum_{j_4 \geq [\alpha k_4]} 2^{-j_4/2} \sum_{j_1,j_2,j_3 \geq [\alpha k_1^*]} \Vert 1_{D_{k_4,\leq j_4}} (f_{k_1,j_1} * f_{k_2,j_2} * f_{k_3,j_3})^{\mathfrak{N}} \Vert_{L^2_{\tau} \ell^2_n} \\
&\lesssim_\varepsilon 2^{-(1/2-\varepsilon) k_4} \prod_{i=1}^3 \sum_{j_i \geq [\alpha k_1^*]} 2^{j_i/2} \Vert f_{k_i,j_i} \Vert_{L^2_{\tau} \ell^2_n}.
\end{split}
\end{equation}
We use \eqref{eq:multilinearL6StrichartzEstimate} to find
\begin{equation*}
\Vert 1_{D_{k_4, \leq j_4}} (f_{k_1,j_1} * f_{k_2,j_2} * f_{k_3,j_3})^{\mathfrak{N}} \Vert_{L^2_{\tau} \ell^2_n} \lesssim_{\varepsilon} 2^{-j_1^*/2} 2^{\varepsilon k_1^*}  2^{j_4/2} \prod_{i=1}^3 2^{j_i/2} \Vert f_{k_i,j_i} \Vert_{L^2_{\tau} \ell^2_n}.
\end{equation*}

We find from the resonance relation that $j_1^* \geq 3 k_1^*-15$.\\
Now the estimate follows in a similar spirit to the computation above. Splitting up the sum over $j_4$ into $[\alpha k_4] \leq j_4 \leq 3k_1^*$ and $j_4 \geq 3k_1^*$, we find
\begin{equation*}
\begin{split}
&\quad 2^{k_4} \sum_{[\alpha k_4] \leq j_4 \leq 3k_1^*} 2^{-j_4/2} \sum_{\substack{j_i \geq [\alpha k_1^*], \\ i =1,2,3}} 2^{(\varepsilon k_1^* /2)} 2^{-3k_1^*/2} 2^{j_4/2} \prod_{i=1}^3 2^{j_i/2} \Vert f_{k_i,j_i} \Vert_{L^2_{\tau} \ell^2_n} \\
&\lesssim_{\varepsilon} k_1^* 2^{-k_1^*/2 + (\varepsilon/2) k_1^*} \prod_{i=1}^3 \sum_{j_i \geq [\alpha k_1^*]} 2^{j_i/2} \Vert f_{k_i,j_i} \Vert_{L^2_{\tau} \ell^2_n} \\
&\lesssim_{\varepsilon} 2^{-(1/2 - \varepsilon)k_4} \prod_{i=1}^3 \sum_{j_i \geq [\alpha k_1^*]} 2^{j_i/2} \Vert f_{k_i,j_i} \Vert_{L^2_{\tau} \ell^2_n}.
\end{split}
\end{equation*}

For the remaining part we argue like above
\begin{equation*}
\begin{split}
&\quad 2^{k_4} \sum_{j_4 \geq 3k_1^*} 2^{-j_4/2} \sum_{{\substack{j_i \geq [\alpha k_1^*], \\ i =1,2,3}}} 2^{(\varepsilon/2) k_1^*} \prod_{i=1}^3 2^{j_i/2} \Vert f_{k_i,j_i} \Vert_{L^2_{\tau} \ell^2_n} \\
&\lesssim 2^{-k_4/2 + \varepsilon k_4}  \prod_{i=1}^3 \sum_{j_i \geq [\alpha k_1^*]} 2^{j_i/2} \Vert f_{k_i,j_i} \Vert_{L^2_{\tau} \ell^2_n},
\end{split}
\end{equation*}
and \eqref{eq:ReductionShorttimeHighHighLowHighEstimate} follows. The variant with one function in a strictly less modulation regularity than $1/2$ follows from the same considerations like in the previous lemma. This finishes the proof. 
\end{proof}
We turn to $High \times High \times High \rightarrow High$-interaction, where we do not use a multilinear argument, but only the bilinear estimate from Lemma \ref{lem:bilinearEstimate}. In the special case $\alpha=1$, this is precisely the analysis from \cite{Molinet2012}. The computation additionally points out that this interaction can be estimated in negative Sobolev spaces for $\alpha>1$. 
\begin{lemma}
\label{lem:ShorttimeHighHighHighHighInteraction}
Let $k_4 \geq 50$ and $|k_i - k_1| \leq 20$ for any $i =2,3,4$ and suppose that $P_{k_i} u_i = u_i$ for $i =1,2,3$. Then, we find \eqref{eq:shorttimeFrequencyLocalizedEstimate} to hold with $D(\alpha,\underline{k}) = 2^{-(\alpha/2-1/2)k_4}$ whenever $\alpha  \geq 1$.
\end{lemma}
\begin{proof}
The usual reduction steps lead us to the remaining estimate
\begin{equation*}
\begin{split}
&\quad \sum_{j_4 \geq [\alpha k_4]} 2^{-j_4/2} 2^{k_4} \sum_{j_1,j_2,j_3 \geq [\alpha k_1^*]} \Vert 1_{D_{k_4,\leq j_4}} (f_{k_1,j_1} * f_{k_2,j_2} * f_{k_3,j_3})^{\mathfrak{N}} \Vert_{L^2_{\tau} \ell^2_n} \\
&\lesssim 2^{-(\alpha/2-1/2)k_4} \prod_{i=1}^3 \sum_{j_i \geq [\alpha k_1^*]} 2^{j_i/2} \Vert f_{k_i,j_i} \Vert_{L^2_{\tau} \ell^2_n}.
\end{split}
\end{equation*}

We use duality to write
\begin{equation*}
\Vert 1_{D_{k_4,\leq j_4}} (f_{k_1,j_1} * f_{k_2,j_2} * f_{k_3,j_3}) \Vert_{L^2_{\tau} \ell^2_n} = \sup_{\Vert u_4 \Vert_{L^2_{t,x} = 1}} \int \int u_1 u_2 u_3 u_4 dx dt,
\end{equation*}
where $u_i = \mathcal{F}_{t,x}^{-1}[f_{k_i,j_i}]$ for $i=1,2,3$.

After splitting the expression according to $P_{\pm} u_i$, where $P_{\pm}$ projects to only positive, respectively negative frequencies, it is easy to see that two bilinear estimates are applicable.\\
Indeed, the same sign must appear twice, which is amenable to \eqref{eq:BilinearEstimate} as the output frequency must be of size $2^{k_1^*}$, and the two remaining factors are also amenable to a bilinear estimate.\\
Say we can apply bilinear estimates to $u_4 u_1$ and $u_2 u_3$. This gives
\begin{equation*}
\begin{split}
&\quad \Vert 1_{D_{k_4,\leq j_4}} (f_{k_1,j_1} * f_{k_2,j_2} * f_{k_3,j_3}) \Vert_{L^2_{\tau} \ell^2_n} \\
&\lesssim 2^{j_1/2} 2^{(j_4-k_4)/4} 2^{j_2/2} 2^{(j_3-k_4)/4} \prod_{i=1}^3 \Vert f_{k_i,j_i} \Vert_{L^2_{\tau} \ell^2_n} \\
&\lesssim 2^{-k_4/2} 2^{j_4/4} 2^{-\alpha k_4/4} \prod_{i=1}^3 2^{j_1/2} \Vert f_{k_i,j_i} \Vert_{L^2_{\tau} \ell^2_n}.
\end{split}
\end{equation*}
The claim follows after summation over $j_4$. The proof is complete. 
\end{proof}
Next, we deal with $High \times High \times Low \rightarrow Low$-interaction:
\begin{lemma}
\label{lem:ShorttimeHighHighLowLowInteraction}
Let $k_3 \geq 20$, $k_1 \leq k_2 \leq k_3$, $k_1 \leq k_2 -5$, $k_4 \leq k_2-5$ and suppose that $P_{k_i} u_i = u_i$ for $i =1,2,3$. Then, we find \eqref{eq:shorttimeFrequencyLocalizedEstimate} to hold with $D(\alpha,\underline{k}) = 2^{(\alpha/2-1+\varepsilon)k_1} 2^{(1-\alpha)k_4}$ for any $\varepsilon > 0$.
\end{lemma}
\begin{proof}
Contrary to the previous cases, we have to add localization in time to estimate $u_{k_i}$ in $F^{\alpha}_{k_i}$ for $i=2,3$.\\
For this purpose let $\gamma: \R \rightarrow [0,1]$ be a smooth function supported in $[-1,1]$ with the property
\begin{equation*}
\sum_{m \in \Z} \gamma^3(x-m) \equiv 1.
\end{equation*}
We find the left-hand side to be dominated by
\begin{equation*}
\begin{split}
\sum_{|m| \lesssim 2^{\alpha(k_1-k_4)}} C 2^{k_4} \sup_{t_{k_4} \in \R} \Vert \mathcal{F}_{t,x}[u_1 \eta_0(2^{\alpha k_4}(t-t_{k_4})) \gamma(2^{\alpha k_1}(t-t_{k_4})-m)] * \\
\mathcal{F}_{t,x}[u_2 \gamma(2^{\alpha k_1}(t-t_{k_4})-m)] * \mathcal{F}_{t,x}[u_3 \gamma(2^{\alpha k_1}(t-t_{k_4})-m)] \Vert_{X_{k_4}}.
\end{split}
\end{equation*}

With the additional localization in time available, we can annex the modulations for $j_i \leq [\alpha k_1^*], \; i=1,2,3$ and denote $f_{k_i} = \mathcal{F}_{t,x}[u_i \gamma(2^{k_1}(t-t_{k_4})-m)]$. Additional localization is denoted by
\begin{equation*}
f_{k_i,j_i} = 
\begin{cases}
	\eta_{\leq j_i}(\tau - n^3) f_{k_i}, \; j_i = [\alpha k_1^*], \\
	\eta_{j_i}(\tau - n^3) f_{k_i}, \; j_i > [\alpha k_1^*].
\end{cases}
\end{equation*}
By the above reductions, we have to prove
\begin{equation*}
\begin{split}
&\quad 2^{\alpha(k_3-k_4)} 2^{k_4} \sum_{j_4 \geq [\alpha k_4]} 2^{-j_4/2} \sum_{j_1,j_2,j_3 \geq [\alpha k_1^*]} \Vert 1_{D_{k_4,\leq j_4}} (f_{k_1,j_1} * f_{k_2,j_2} * f_{k_3,j_3})^{\mathfrak{N}} \Vert_{L^2_{\tau} \ell^2_n} \\
&\lesssim 2^{(\alpha/2-1+\varepsilon)k_3} 2^{(1-\alpha)k_4} \prod_{i=1}^3 \sum_{j_i \geq [\alpha k_1^*]} 2^{j_i/2} \Vert f_{k_i,j_i} \Vert_{L^2_{\tau} \ell^2_n}.
\end{split}
\end{equation*}

As in the proof of Lemma \ref{lem:ShorttimeHighLowLowHighInteraction}, the resonance is localized to
\begin{equation*}
2^{2k_1^*} \lesssim |\Omega| \lesssim 2^{2k_1^*+k_3^*},
\end{equation*}
and we introduce additional localization $P^l_{\Omega}$ for $|\Omega| \sim 2^{2k_1^*+l}$, where $0 \leq l \leq k_3^*$. Correspondingly, we decompose $f_{k_1,j_1}$ into intervals $I$ of length $2^l$, which allows an almost orthogonal decomposition of the output. At this point, by convolution constraint and almost orthogonality, we can suppose that $\text{supp}_n(f_{k_2,j_2})$ and $\text{supp}_n(f_{k_3,j_3})$ are intervals of length $2^l$.\\
Lastly, we split the sum over $j_4$ into $j_4 \leq 2k_1^* + l$ and $j_4 \geq 2k_1^* + l$. For fixed $l$ we find from \eqref{eq:HighLowLowHighInteractionEstimate}
\begin{equation*}
\begin{split}
&\quad 2^{\alpha (k_3-k_4)} 2^{k_4} \sum_{[\alpha k_4] \leq j_4 \leq 2k_1^* + l} \sum_{\substack{j_i \geq [\alpha k_1^*], \\ i =1,2,3}} 2^{-j_4/2} \\
&\quad \quad \left( \sum_{I_1,I_4} \Vert P^l_{\Omega} 1_{D^{I_4}_{k_4,\leq j_4}} (f^{I_1}_{k_1,j_1} * f_{k_2,j_2} * f_{k_3,j_3} )^{\mathfrak{N}} \Vert^2_{L^2_{\tau} \ell^2_n} \right)^{1/2} \\
&\lesssim 2^{\alpha k_3} 2^{(1-\alpha)k_4} \sum_{[\alpha k_4] \leq j_4 \leq 2k_1^* +l} 2^{-j_4/2} \sum_{\substack{j_i \geq [\alpha k_1^*], \\ i =1,2,3}} 2^{-j_1^*/2} 2^{l/2} 2^{-[\alpha k_3]/2} 2^{j_4/2} \\
&\quad \quad \prod_{i=1}^3 2^{j_i/2} \Vert f_{k_i,j_i} \Vert_2 \\
&\lesssim k_1^* 2^{\alpha k_3/2} 2^{(1-\alpha)k_4} 2^{-k_3} \prod_{i=1}^3 \sum_{j_i \geq [\alpha k_1^*]} 2^{j_i/2} \Vert f_{k_i,j_i} \Vert_2 \\
&\lesssim_{\varepsilon} 2^{(\alpha/2-1+\varepsilon/2) k_3} 2^{(1-\alpha)k_4} \prod_{i=1}^3 \sum_{j_i \geq [\alpha k_1^*] } 2^{j_i/2} \Vert f_{k_i,j_i} \Vert_{L^2}.
\end{split}
\end{equation*}

Likewise we find for the contribution of $j_4 \geq 2k_1^* + l$ the bound
\begin{equation*}
\lesssim 2^{(\alpha/2-1+\varepsilon/2)k_1} 2^{(1-\alpha)k_4} \prod_{i=1}^3 \sum_{j_i \geq [\alpha k_1^*]} 2^{j_i/2} \Vert f_{k_i,j_i} \Vert_2.
\end{equation*}
Summation over $l$ yields the claim. 
\end{proof}
At last, we turn to $High \times High \times High \rightarrow Low$-interaction:
\begin{lemma}
\label{lem:ShorttimeHighHighHighLowInteraction}
Let $k_1 \geq 50$, $|k_1-k_2| \leq 10$, $|k_1-k_3| \leq 10$, $k_4 \leq k_1-20$ and suppose that $P_{k_i} u_i = u_i$ for $i = 1,2,3$. Then, we find \eqref{eq:shorttimeFrequencyLocalizedEstimate} to hold with $D(\alpha,\underline{k}) = 2^{(\alpha-3/2+\varepsilon)k_1} 2^{(1-\alpha)k_4}$ for any $\varepsilon > 0$.
\end{lemma}
\begin{proof}
Like in Lemma \ref{lem:ShorttimeHighHighLowLowInteraction} we have to add localization in time according to $k_1^*$. By the notation and conventions from above, we have to show the estimate
\begin{equation*}
\begin{split}
&\quad 2^{k_4} 2^{\alpha(k_1-k_4)} \sum_{j_4 \geq [\alpha k_4]} 2^{-j_4/2} \sum_{\substack{j_i \geq [\alpha k_1^*], \\ i =1,2,3}} \Vert 1_{D_{k_4,\leq j_4}}(f_{k_1,j_1} * f_{k_2,j_2} * f_{k_3,j_3})^{\mathfrak{N}} \Vert_{L^2_{\tau} \ell^2_n} \\
&\lesssim_{\varepsilon} 2^{(1-\alpha)k_4} 2^{(\alpha - 3/2 + \varepsilon)k_1} \prod_{i=1}^3 \sum_{j_i \geq [\alpha k_1^*]} 2^{j_i/2} \Vert f_{k_i,j_i} \Vert_{L^2_{\tau} \ell^2_n}.
\end{split}
\end{equation*}
The resonance function implies $j_1^* \geq 3k_1^* - 20$. We split the sum over $j_4$ into $[\alpha k_4] \leq j_4 \leq 3k_1^*$ and $j_4 \geq 3k_1^*$.
The first part is estimated by an application of \eqref{eq:multilinearL6StrichartzEstimate}
\begin{equation*}
\begin{split}
&\quad 2^{k_4} 2^{\alpha (k_1 - k_4)} \sum_{[\alpha k_4] \leq j_4 \leq 3k_1^*} 2^{-j_4/2} \sum_{\substack{j_i \geq [\alpha k_1^*], \\ i =1,2,3}} \Vert 1_{D_{k_4,\leq j_4}} (f_{k_1,j_1} * f_{k_2,j_2} * f_{k_3,j_3})^{\mathfrak{N}} \Vert_{L^2_{\tau} \ell^2_n} \\
&\lesssim_{\varepsilon} 2^{\alpha k_1} 2^{(1-\alpha)k_4} \sum_{[\alpha k_4] \leq j_4 \leq 3k_1^*} 2^{-j_4/2} \sum_{\substack{j_i \geq [\alpha k_1^*], \\ i =1,2,3}} 2^{-j_1^*/2} 2^{(\varepsilon k_1^*)/2} 2^{j_4/2}  \prod_{i=1}^3 2^{j_i/2} \Vert f_{k_i,j_i} \Vert_{L^2_{\tau} \ell^2_n} \\
&\lesssim_{\varepsilon} 2^{(\alpha+\varepsilon/2-3/2) k_1} 2^{k_4(1-\alpha)} (3k_1^*) 2^{(\alpha-2)k_1/2} \prod_{i=1}^3 \sum_{j_i \geq [\alpha k_1^*]} 2^{j_i/2} \Vert f_{k_i,j_i} \Vert_{L^2_{\tau} \ell^2_n}.
\end{split}
\end{equation*}
The estimate for $j_4 \geq 3k_1^*$ follows similarly, which proves the claim together with the standard modification of lowering the modulation regularity slightly. 
\end{proof}
For all frequencies low we have the following trivial estimate due to the Cauchy-Schwarz inequality.
\begin{lemma}
\label{lem:LowLowLowLowInteraction}
Let $k_1,\ldots,k_4 \leq 200$. Then, we find \eqref{eq:shorttimeFrequencyLocalizedEstimate} to hold with $D(\alpha,\underline{k}) = 1$.
\end{lemma}
We summarize the lower regularity thresholds, for which we can show the trilinear estimate \eqref{eq:shorttimeTrilinearEstimate} by splitting up the frequencies and using the estimate \eqref{eq:shorttimeFrequencyLocalizedEstimate}:
\begin{enumerate}
\item $High \times Low \times Low \rightarrow High$-interaction: Lemma \ref{lem:ShorttimeHighLowLowHighInteraction} provides us with the regularity threshold $s=-(\alpha/4)+$.
\item $High \times High \times Low \rightarrow High$-interaction: Lemma \ref{lem:ShorttimeHighHighLowHighInteraction} provides us with the regularity threshold $s=-(\alpha/4)+$.
\item $High \times High \times High \rightarrow High$-interaction: Lemma \ref{lem:ShorttimeHighHighHighHighInteraction} provides us with the regularity threshold $s=(1-\alpha)/4$.
\item $High \times High \times Low \rightarrow Low$-interaction: Lemma \ref{lem:ShorttimeHighHighLowLowInteraction} provides us with the regularity threshold $s=-(1/6)+$ for $\alpha = 1$.
\item $High \times High \times High \rightarrow Low$-interaction: Lemma \ref{lem:ShorttimeHighHighHighLowInteraction} provides us with the regularity threshold $s=-(1/6)+$ for $\alpha = 1$.
\item $Low \times Low \times Low \rightarrow Low$-interaction: By Lemma \ref{lem:LowLowLowLowInteraction}, there is no threshold.
\end{enumerate}
We have proved the following proposition:
\begin{proposition}
\label{prop:nonlinearEstimate}
Let $T \in (0,1]$. For $0<s<1/2$, there is $\alpha(s)<1$ and $\theta = \theta(s)>0$ or $s=0$, $\alpha=1$ and $\theta=0$ such that
\begin{equation*}
\Vert \mathfrak{N}(u_1,u_2,u_3) \Vert_{N^{s,\alpha}(T)} \lesssim T^\theta \prod_{i=1}^3 \Vert u_i \Vert_{F^{s,\alpha}(T)}.
\end{equation*}
Furthermore, there is $\delta^\prime > 0$ so that for any $0<\delta <\delta^\prime$ there is $s=s(\delta)<0$ and $\theta > 0$ such that
\begin{equation*}
\Vert \mathfrak{N}(u_1,u_2,u_3) \Vert_{N^{s,1+\delta}(T)} \lesssim T^\theta \prod_{i=1}^3 \Vert u_i \Vert_{F^{s,1+\delta}(T)}.
\end{equation*}
\end{proposition}
\section{Energy estimates}
\label{section:EnergyEstimates}
We have to propagate the energy norm to finish the proofs of Theorems \ref{thm:localExistence} and \ref{thm:nonExistence}. This is achieved in the following proposition:
\begin{proposition}
\label{prop:energyPropagation}
\begin{enumerate}
\item[(a)]
Suppose that $\alpha=1$. There is $\theta(s)>0$ so that we find the following estimate to hold
\begin{equation*}
\Vert u \Vert_{E^s(T)}^2 \lesssim \Vert u_0 \Vert_{H^s}^2 + T^\theta \Vert u \Vert_{F^{s,\alpha}(T)}^4
\end{equation*}
whenever $s>1/4$. Furthermore, there are non-negative functions $c(s),d(s)$ and $\theta(s)>0$ so that we find for any $M \in 2^{\mathbb{N}}$ the estimate
\begin{equation}
\label{eq:energyPropagationB}
\Vert u \Vert_{E^s(T)}^2 \lesssim \Vert u_0 \Vert_{H^s}^2 + T^{\theta} M^{c(s)} \Vert u \Vert_{F^{s,\alpha}(T)}^4 + M^{-d(s)} \Vert u \Vert_{F^{s,\alpha}(T)}^4 + T^\theta \Vert u \Vert_{F^{s,\alpha}(T)}^6
\end{equation}
to hold whenever $s>0$.
\item[(b)]
Suppose that \eqref{eq:L8Hypothesis} is true. Then, there is $s^\prime<0$ so that for $s^\prime<s<0$ there is $\delta(s)>0$ and there are non-negative functions $c(s),d(s)$ and $\theta(s)>0$ so that \eqref{eq:energyPropagationB} holds true for $\alpha = 1 + \delta$.
\end{enumerate}
\end{proposition}
In Subsection \ref{subsection:EnergyEstimatesPositiveSobolevSpaces} we derive estimates in positive Sobolev spaces for the proof of part (a). In Subsection \ref{subsection:EnergyEstimatesNegativeSobolevSpaces} we make use of the conjectured $L^8_{t,x}$-Strichartz estimate to propagate the energy norm in negative Sobolev spaces. This yields the necessary estimates for the proof of part (b). In Subsection \ref{subsection:ConclusionEnergyEstimate} the proof of Proposition \ref{prop:energyPropagation} is concluded.
\subsection{Energy estimates in positive Sobolev spaces}
\label{subsection:EnergyEstimatesPositiveSobolevSpaces}
To prove the above estimates, we analyze the energy functional 
\begin{equation*}
\Vert u(t) \Vert_{H^s}^2 = \sum_{k \in \Z} \langle k \rangle^{2s} |\hat{u}(t,k)|^2.
\end{equation*}
It turns out that control over the Sobolev norm is not enough to control the norm of the energy space because the norm of $E^s(T)$ differs from the $H^s$-norm by a logarithm. The remedy is to control a slightly larger class of symbols.

Symbols of the following kind can already be found in \cite{KochTataru2007}, see also \cite[Section~2.3.,~p.~15]{OhWang2018} for a more constructive description of related symbols.
\begin{definition}
Let $\varepsilon > 0$ and $s \in \mathbb{R}$. Then $S^s_{\varepsilon}$ is the set of positively real-valued, spherically symmetric and smooth functions (symbols) with the following properties:
\begin{enumerate}
\item[(i)] Slowly varying condition: For $\xi \sim \xi^\prime$ we have
\begin{equation*}
a(\xi) \sim a(\xi^\prime),
\end{equation*}
\item[(ii)] symbol regularity,
\begin{equation*}
|\partial^{\alpha} a(\xi)| \lesssim_{\alpha} a(\xi) (1+\xi^2)^{-\alpha/2},
\end{equation*}
\item[(iii)] growth at infinity, for $|\xi| \gg 1$ we have
\begin{equation*}
s - \varepsilon \leq \frac{\log a(\xi)}{\log(1+\xi^2)} \leq s+ \varepsilon.
\end{equation*}
\end{enumerate}
\end{definition}
It will be admissible to choose $\varepsilon = \varepsilon(s)>0$ in the following, but the subsequent estimates must be uniform in $\varepsilon$. We also write 
\begin{equation}
\label{eq:SymbolDomination}
\tilde{a}(2^{k}) = 2^{2k(s+\varepsilon)}
\end{equation}
 because the expression safely estimates combinations of $a(2^k)$.
 
For $a \in S^{s}_\varepsilon$ we set
\begin{equation*}
\Vert u(t) \Vert_{H^a}^2 = \sum_{n \in \Z} a(n) |\hat{u}(t,n)|^2,
\end{equation*}
and, for a real-valued solution to \eqref{eq:renormalizedmKdV}, we compute
\begin{equation*}
\begin{split}
\partial_t \Vert u(t) \Vert^2_{H^a} &= 2 \Re ( \sum_{n \in \Z} a(n) \partial_t \hat{u}(t,n) \hat{u}(t,-n) ) \\
&= 2 \Re ( \sum_{n \in \Z} a(n) i n^3 |\hat{u}(t,n)|^2 + in a(n) |\hat{u}(t,n)|^2 \hat{u}(t,n) \hat{u}(t,-n) \\
&\quad \quad + i\frac{n}{3} a(n) \sum_{\substack{n = n_1+n_2+n_3,\\ (*)}} \hat{u}(t,n_1) \hat{u}(t,n_2) \hat{u}(t,n_3) \hat{u}(t,-n) ) \\
&= C \Re ( \sum_{\substack{n_1+\ldots+n_4 = 0,\\ (*)}} \left( a(n_1) n_1 + a(n_2) n_2 + a(n_3) n_3 + n_4 a(n_4) \right) \\
 &\quad \quad \hat{u}(t,n_1) \hat{u}(t,n_2) \hat{u}(t,n_3) \hat{u}(t,n_4) ).
\end{split}
\end{equation*}
The last step follows from a symmetrization argument, which fails for the difference equation. This is due to the lack of continuous dependence for $s<1/2$.

The fundamental theorem of calculus yields
\begin{equation*}
\Vert u(t) \Vert^2_{H^a} = \Vert u_0 \Vert_{H^a}^2 + C R_4^{s,a,M}(t,u) + C B_4^{s,a,M}(t,u) + C R_6^{s,a,M}(t,u).
\end{equation*}
The expressions are explained in detail below. The necessary estimates to deduce Proposition \ref{prop:energyPropagation} from the above display are carried out in Lemma \ref{lem:estimateBoundaryTerm} and Propositions \ref{prop:symmetrizedEnergyEstimate} and \ref{prop:EstimateCorrectionTerm}.

We turn to the details: In the following denote $\overline{n} = (n_1,\ldots,n_4)$ for $n_1+\ldots+n_4=0$. We set
\begin{equation*}
\psi_{s,a}(\overline{n}) = \sum_{i=1}^4 a(n_i) n_i
\end{equation*}
and
\begin{equation*}
R^{s,a}_4(T,u_1,\ldots,u_4) = \sum_{\substack{n_1+n_2+n_3+n_4=0,\\(*)}} \int_0^{T}  \psi_{s,a}(\overline{n}) \prod_{i=1}^4 \hat{u}_i(t,n_i) dt.
\end{equation*}
Write $R^{s,a}_4(T,u) := R^{s,a}_4(T,u,u,u,u)$.\\
We found 
\begin{equation*}
\Vert u(t) \Vert_{H^a}^2 = \Vert u_0 \Vert_{H^a}^2 + C R^{s,a}_4(t,u),
\end{equation*}
 and we shall see that the above expression can be estimated as long as $s>1/4$ in $F^{s,1}(T)$-spaces.
 
 To go below $s=1/4$ to $L^2$, we add a correction term in a similar spirit to the $I$-method (see e.g. \cite{CollianderKeelStaffilaniTakaokaTao2003}). But the boundary term is insensitive to the length of the time interval. To remedy this, we do not differentiate by parts all of $R^{s,a}_4$, but only the part, which contains high frequencies.
 
More precisely, we set for a large frequency $M \in 2^{\mathbb{N}}$
\begin{equation*}
R^{s,a,M}_4(T,u) = C \int_0^T \sum_{\substack{n_1+\ldots+n_4=0,\\(*),|n_j| \leq M}} \psi_{s,a}(\overline{n}) \prod_{i=1}^4 \hat{u}(t,n_i) dt
\end{equation*}
and decompose 
\begin{equation*}
R^{s,a}_4(t,u) = R^{s,a,M}_4(t,u) + (R^{s,a}_4(t,u) - R_4^{s,a,M}(t,u)).
\end{equation*}
The frequency cutoff $M$ will be later chosen in dependence of the norm of the initial value.

Next, we differentiate by parts, but only the term
\begin{equation*}
R^{s,a}_4(t,u) - R_4^{s,a,M}(t,u) = B_4^{s,a,M}(t,u) + R_6^{s,a,M}(t,u).
\end{equation*}
We have
\begin{equation*}
\partial_t \hat{u}(t,n) + (in)^3 \hat{u}(t,n) = in |\hat{u}(t,n)|^2 \hat{u}(t,n) + \frac{in}{3} \sum_{\substack{n=n_1+n_2+n_3,\\(*)}} \hat{u}(t,n_1) \hat{u}(t,n_2) \hat{u}(t,n_3).
\end{equation*}
After changing to interaction picture $\hat{v}(t,n) = e^{-i n^3 t} \hat{u}(t,n)$, we find for solutions $u$
\begin{equation*}
\partial_t \hat{v}(t,n) = in |\hat{v}(n)|^2 \hat{v}(n) + \frac{in}{3} \sum_{\substack{n= n_1+n_2+n_3,\\ (*)}} e^{i t \Omega(\overline{n})} \hat{v}(t,n_1) \hat{v}(t,n_2) \hat{v}(t,n_3).
\end{equation*}
In this context, the resonance function is given by
\begin{equation*}
\Omega(\overline{n}) = \sum_{i=1}^4 n_i^3 = -3(n_1+n_2)(n_1+n_3)(n_2+n_3), \quad n_4 = -n.
\end{equation*}
Differentiation of $R^{s,a,M}_4$ by parts is possible because the resonance function does not vanish for the terms in $R_4^{s,a}$:
\begin{equation*}
\begin{split}
R_4^{s,a}(T,u) &= \int_0^T dt \sum_{\substack{n_1+\ldots+n_4=0,\\(*)}} \psi_{s,a}(\overline{n}) \prod_{i=1}^4 \hat{u}(t,n_i) \\
&= \sum_{\substack{n_1+\ldots+n_4=0,\\(*)}} \psi_{s,a}(\overline{n}) \int_0^T dt e^{i t \Omega(\overline{n})} \prod_{i=1}^4 \hat{v}(t,n_i) \\
&= \sum_{\substack{n_1+\ldots+n_4=0,\\(*)}} \psi_{s,a}(\overline{n}) \int_0^T dt \partial_t \left( \frac{e^{i t \Omega(\overline{n})}}{i \Omega(\overline{n})} \right) \prod_{i=1}^4 \hat{v}(t,n_i) \\
&= \left[ \sum_{\substack{n_1+\ldots+n_4=0,\\(*)}} \frac{\psi_{s,a}(\overline{n})}{i \Omega(\overline{n})} \prod_{i=1}^4 \hat{u}(t,n_i) \right]_{t=0}^T  \\
&\quad + 4 \sum_{\substack{n_1+n_2+n_3+n_4=0,\\(*)}} \frac{\psi_{s,a}(\overline{n})}{i \Omega(\overline{n})} \int_0^T dt (\partial_t \hat{v}(t,n_1)) \prod_{i=2}^4 \hat{v}(t,n_i).
\end{split}
\end{equation*}

Set
\begin{align*}
B_4^{s,a}(T,u) &= \left[ \sum_{\substack{n_1+\ldots+n_4=0,\\(*)}} \frac{\psi_{s,a}(\overline{n})}{i \Omega(\overline{n})} \prod_{i=1}^4 \hat{u}(t,n_i) \right]_{t=0}^T, \\
I(T,u) &= C \int_0^T dt \sum_{\substack{n_1+\ldots+n_4=0,\\(*)}} \frac{\psi_{s,a}(\overline{n}) n_1}{\Omega(\overline{n})} |\hat{u}(t,n_1)|^2 \hat{u}(t,n_1) \prod_{i=2}^4 \hat{u}(t,n_i), \\
II(T,u) &= C \int_0^T dt \sum_{\substack{n_1+\ldots+n_4 = 0,\\ (*)}} \frac{\psi_{s,a}(\overline{n}) n_1}{\Omega(\overline{n})} \prod_{i=2}^4 \hat{u}(t,n_i)  \\
 &\quad \quad \sum_{\substack{n_1+n_5+n_6+n_7 = 0,\\(*)}} \prod_{i=5}^7 \hat{u}(t,n_i).
\end{align*}
If we differentiate only $R^{s,a}_4(t,u) - R_4^{s,a,M}(t,u)$, then one of the initial frequencies has to be larger than $M$.

The following lemma provides us with a useful pointwise bound on $|\psi_{s,a}|$. Recall the notation \eqref{eq:SymbolDomination} to dominate expressions involving the symbol $a$.
\begin{lemma}
\label{lem:pointwiseMultiplierBound}
Let $s>0,\;0<\varepsilon<s$ and $a \in S^s_{\varepsilon}$. Suppose that $n_i \in I_{k_i}$ for $i =1,\ldots,4$. Then, we find the following estimate to hold:
\begin{equation}
\label{eq:pointwiseMultiplierBound}
|\psi_{s,a}(\overline{n})| \lesssim \frac{\tilde{a}(2^{k_1^*})}{2^{2k_1^*}} |\Omega(\overline{n})|
\end{equation}
\end{lemma}
The tools, which we use to derive the pointwise bound, are the mean value theorem and the double mean value theorem. To avoid confusion, we recall the double mean value theorem.
\begin{lemma}
\label{lem:doubleMeanValueTheorem}
If $y$ is controlled by $z$ and $|\eta|, |\lambda| \ll |\xi|$, then
\begin{equation*}
y(\xi+\eta+\lambda) - y(\xi+\eta) - y(\xi+\lambda) + y(\xi) = \mathcal{O}(|\eta||\lambda| \frac{z(\xi)}{\xi^2}).
\end{equation*}
\end{lemma}
\begin{proof}
Cf. \cite[Lemma~4.2.,~p.~715]{CollianderKeelStaffilaniTakaokaTao2003}. 
\end{proof}
We are ready to prove Lemma \ref{lem:pointwiseMultiplierBound}.
\begin{proof}[Lemma \ref{lem:pointwiseMultiplierBound}]
We prove the bound through Case-by-Case analysis.\\
\textbf{Case 1:} $|n_1| \sim |n_2| \sim |n_3| \sim |n_4| \sim 2^{k_1^*}$.

\underline{Subcase a:} Two of the factors $|n_2+n_3|, |n_2+n_4|, |n_3+n_4|$ are much smaller than $2^{k_1^*}$ (note that one factor must be of size $2^{k_1^*}$ because at least two numbers are of the same sign).\\
For definiteness suppose in the following that $|n_2+n_3| \ll 2^{k_1^*}, \; |n_1+n_2| \ll 2^{k_1^*}$, and from this assumption follows $|\Omega(\overline{n})| \sim 2^{k_1^*} |n_2+n_3| |n_1+n_2|$.

We set $\xi = n_1$, $\xi + \eta = -n_2$, $\xi+ \lambda = -n_4$, $\xi+\eta+\lambda = n_3$, to check that the assumptions of the double-mean value theorem for the function $a(\cdot) \cdot$ are fulfilled. Hence, by property (ii) of the symbol, we find $|\psi_{s,a}| \lesssim \frac{\tilde{a}(2^{k_1^*})}{2^{k_1^*}} |n_2+n_3| |n_1+n_2|$. Consequently, we find \eqref{eq:pointwiseMultiplierBound} to hold in this case.

\underline{Subcase b:} Next, suppose that one of the factors $|n_2+n_3|, |n_2+n_4|, |n_3+n_4|$ is much smaller than $2^{k_1^*}$, whereas the others are comparable to $2^{k_1^*}$. By symmetry suppose that this is $|n_1+n_2|$.\\
For the resonance function follows $|\Omega(\overline{n})| \sim 2^{2k_1^*} |n_1+n_2| $. We invoke the mean value theorem to find
\begin{equation*}
|\psi_{s,a}(\overline{n})| \lesssim \tilde{a}(2^{k_1^*}) |n_1+n_2| \sim \frac{\tilde{a}(2^{k_1^*})}{2^{2k_1^*}} |\Omega(\overline{n})|,
\end{equation*}
which proves the claim in this subcase.

\underline{Subcase c:} Suppose that all three factors $|n_2+n_3|, |n_2+n_4|, |n_3+n_4|$ are comparable to $2^{k_1^*}$. This gives $|\Omega(\overline{n})| \sim 2^{3k_1^*}$. By the trivial bound 
\begin{equation}
\label{eq:TrivialBoundPsi}
|\psi_{s,a}(\overline{n})| \lesssim \tilde{a}(2^{k_1^*}) 2^{k_1^*},
\end{equation}
 we find \eqref{eq:pointwiseMultiplierBound} to hold also in this subcase.
 
\textbf{Case 2:} $|n_1| \ll |n_2| \sim |n_3| \sim |n_4| \sim 2^{k_1^*}$.\\
In this case we have $|\Omega(\overline{n})| \sim 2^{3k_1^*}$. Together with \eqref{eq:TrivialBoundPsi}, the estimate \eqref{eq:pointwiseMultiplierBound} is immediate.

\textbf{Case 3:} $|n_1|, |n_2| \ll |n_3| \sim |n_4| \sim 2^{k_1^*}$.\\
In this case we find $|\Omega(\overline{n})| \sim |n_1+n_2| 2^{2k_1^*} $, and an application of the mean value theorem yields
\begin{equation*}
|\psi_{s,a}(\overline{n})| \lesssim |n_1+n_2| \tilde{a}(2^{k_1^*}),
\end{equation*}
which yields the claim in this case. 
\end{proof}
\begin{remark}
\label{rem:RegularityMultiplierNegativeSobolevSpacesMKDV}
In negative Sobolev spaces a related estimate was proven in\\ \cite[Lemma~5.2,~p.~59]{ChristHolmerTataru2012}. There, also regularity of the extension was proved. This will allow us to separate variables in negative Sobolev spaces.
\end{remark}
The second important ingredient to find the bound for $R_4^{s,a}$ is the following improvement of the $L_{t,x}^6$-Strichartz estimate. We remark that following along the lines of Section \ref{section:MultilinearEstimates} one can derive stronger estimates in some cases. But since this would not improve the overall analysis, we record only the simplified version below. The lemma can be proved like in Section \ref{section:MultilinearEstimates}.
\begin{lemma}
\label{lem:RefinedStrichartzEstimateL6}
Suppose that $\text{supp}(\tilde{u}_i) \subseteq D_{k_i,\leq j_i}$, where $j_i \geq [\alpha k_1^*]$ for $1 \leq \alpha \leq 2$, $i=1,\ldots,4$. Then, we find the following estimate to hold:
\begin{equation*}
\Vert P_{k_1} \mathcal{N}(u_2,u_3,u_4) \Vert_{L^2_{t,x}} \lesssim 2^{-k_1^*/2} 2^{k_4^*/2} \prod_{i=2}^4 2^{j_i/2} \Vert u_i \Vert_{L^2_{t,x}}.
\end{equation*}
\end{lemma}
Having the Lemmata \ref{lem:pointwiseMultiplierBound} and \ref{lem:RefinedStrichartzEstimateL6} at disposal, we can find a bound on $R^{s,a,M}_4$.
\begin{proposition}
\label{prop:symmetrizedEnergyEstimate}
Suppose that $\alpha = 1$. Then, there are functions $\theta(s)>0$, $c(s) \geq 0$ with $c(s)=0$ for $s>1/4$ and $\varepsilon(s)>0$ so that for any $M \in 2^{\N}$ we find the estimate
\begin{equation}
\label{eq:totalRemainderEstimateA}
R^{s,a,M}_4(T,u_1,\ldots,u_4) \lesssim T M^{c(s)} \prod_{i=1}^4 \Vert u_i \Vert_{F^{s,\alpha}(T)}
\end{equation}
to hold provided that $s>-1/2$ and $a \in S^s_\varepsilon$.
\end{proposition}
\begin{proof}
The strategy of the proof is as follows: Firstly, we apply a dyadic decomposition on the spatial frequencies. That is we estimate $R^{s,a,M}_4(T,u_1,\ldots,u_4)$ for frequency localized functions $u_i$, where $P_{k_i} u = u$, $k_i \leq \log_2(M)$. For these functions we will show the estimate
\begin{equation}
\label{eq:localizedRemainderEstimateA}
R^{s,a,M}_4(T,u_1,\ldots,u_4) \lesssim T^{\theta} \prod_{i=1}^4 2^{(s-)k_i} \Vert u_i \Vert_{F^{\alpha}_{k_i}}
\end{equation}
for $s>1/4$. The slightly less regularity than $s$ on the right-hand side allows us to sum over dyadic blocks in the end. With the frequencies being smaller than $M$, from \eqref{eq:localizedRemainderEstimateA} for $s>1/4$ follows already \eqref{eq:totalRemainderEstimateA} for $s>-1/2$.

Next, we localize time antiproportionally to the highest frequency. Let $\gamma: \R \rightarrow [0,1]$ be a smooth function with compact support in $[-1,1]$ and the property
\begin{equation*}
\sum_{m \in \Z} \gamma^4(x-m) \equiv 1.
\end{equation*}
With this function, we write
\begin{equation*}
\begin{split}
R^{s,a}_4(T,u_1,\ldots,u_4) &= \int_0^T dt \sum_{\substack{n_1+\ldots+n_4=0,\\(*)}} \psi_{s,a}(\overline{n}) \prod_{i=1}^4 \hat{u}_i(t,n_i) \\
&= \int_0^T dt \sum_{m \in \Z} \sum_{\substack{n_1+\ldots+n_4=0,\\(*)}} \psi_{s,a}(\overline{n}) \gamma(2^{\alpha k_1^*} t -m) \hat{u}_1(t,n_1) \times \\ 
&\quad \quad \prod_{i=2}^4 \gamma(2^{\alpha k_1^*}t-m) \hat{u}_i(t,n_i) 
\end{split}
\end{equation*}
Already note that there are $\mathcal{O}(T 2^{k_1^*})$ values of $m$, for which the above expression does not vanish.

With $f_{k_i}(\tau,\xi) = \mathcal{F}_{t,x}(\gamma(2^{\alpha k_1^*} t-m) {u}_i)(\tau,\xi)$, we localize modulation\footnote{Strictly speaking, we had to consider $f_{m,k_i}$ or $f_{m,k_i,j_i}$, respectively, tracking the additional dependence on $m$. Since all the estimates below are uniform in $m$, we choose to drop dependence on $m$ for the sake of brevity.}
\begin{equation*}
f_{k_i,j_i} = 
\begin{cases}
	\eta_{\leq j_i}(\tau - n^3) f_{k_i}, \; j_i = [\alpha k_1^*], \\
	\eta_{j_i}(\tau - n^3) f_{k_i}, \; j_i > [\alpha k_1^*].
\end{cases}
\end{equation*}
In the above sum over $m$, in case of nontrivial contribution, we have to distinguish between the two cases:
\begin{equation*}
\begin{split}
\mathcal{A} &= \{ m \in \Z | 1_{[0,T]}(\cdot) \gamma(2^{\alpha k_1^*} \cdot - m) = \gamma(2^{\alpha k_1^*} \cdot - m) \},\\
\mathcal{B} &= \{ m \in \Z | 1_{[0,T]}(\cdot) \gamma(2^{\alpha k_1^*} \cdot - m) \neq \gamma(2^{\alpha k_1^*} \cdot - m) \text{ and } 1_{[0,T]}(\cdot) \gamma(2^{\alpha k_1^*} \cdot - m) \neq 0 \}.
\end{split}
\end{equation*}

Note that $\# \mathcal{B} \leq 4$. Consequently, we save a factor $2^{ k_1^*}$ compared to $\mathcal{A}$, and we only sketch the necessary modifications after treating the cases from $\mathcal{A}$.\\
Therefore, we focus on estimates for $m \in \mathcal{A}$, where $\# \mathcal{A} \lesssim T 2^{k_1^*}$.

Firstly, we estimate $High \times High \times High \times High$-interaction. That means all frequencies are comparable, and we suppose that $|k_1-k_i| \leq 20$ for $i =2,3,4$.\\
Recall the pointwise bound from Lemma \ref{lem:pointwiseMultiplierBound}. To make effective use, we introduce another dyadic sum governing the size of $|\Omega(\overline{n})|$. Below, we take $|\Omega(\overline{n})| \sim 2^k$, where $k_1^* \leq k \leq 3k_1^*$ and sum over $k$ in the end.\\
We observe
\begin{equation}
\label{eq:HighHighHighHighSumEstimate}
\begin{split}
\sum_{\substack{k_1^* \leq k \leq 3k_1^*,\\|\Omega|\sim 2^k}} \sum_{m \in \mathcal{A}} \frac{\tilde{a}(2^{k_1^*})}{2^{2k_1^*}} |\Omega|^{1/2} &\lesssim T \sum_{\substack{k_1^* \leq k \leq 3k_1^*,\\|\Omega|\sim 2^k}} \frac{\tilde{a}(2^{k_1^*})}{2^{k_1^*}} 2^{k/2} \lesssim T \tilde{a}(2^{k_1^*}) 2^{k_1^*/2}\\
&\lesssim T \prod_{i=1}^4 2^{(s-)k_i}
\end{split}
\end{equation}
provided that $s>1/4$.

Further, we suppose due to symmetry that $j_1 = j_1^*$. Together with the resonance relation $2^{j_1^*} \gtrsim |\Omega(\overline{n})| \sim 2^k$ and Lemma \ref{lem:RefinedStrichartzEstimateL6}, we find
\begin{equation}
\label{eq:EnergyEstimateWithResonance}
\begin{split}
&\quad \int_{\tau_1+\ldots+\tau_4=0} \sum_{\substack{ n_1+\ldots+n_4=0,\\(*)}} \prod_{i=1}^4 f_{k_i,j_i}(\tau_i,n_i) \\
&\lesssim 2^{j_1/2} \Vert f_{k_1,j_1} \Vert_{L_\tau^2 \ell^2_n} 2^{-k/2} \Vert 1_{I_{k_1}} (f_{k_2,j_2}*f_{k_3,j_3}*f_{k_4,j_4})^{\mathfrak{N}} \Vert_{L_\tau^2 \ell_n^2} \\
&\lesssim 2^{-k/2} \prod_{i=1}^4 2^{j_i/2} \Vert f_{k_i,j_i} \Vert_{L^2_{\tau} \ell^2_n},
\end{split}
\end{equation}
and \eqref{eq:localizedRemainderEstimateA} follows from \eqref{eq:HighHighHighHighSumEstimate} and \eqref{eq:EnergyEstimateWithResonance} due to \eqref{eq:XkEstimateIII}.

We turn to the cases described by the set $\mathcal{B}$. We have to estimate the expression
\begin{equation*}
\frac{\tilde{a}(2^{k_1^*})}{2^{2k_1^*}} \int_{\tau_1+\ldots+\tau_4=0} \sum_{\substack{n_1+\ldots+n_4=0,\\(*)}} |\Omega(\overline{n})| \prod_{i=1}^4 f_{k_i,j_i}(\tau_i,n_i),
\end{equation*}
where 
\begin{equation*}
f_{k_1}(\tau,n) = \mathcal{F}_{t,x}[1_{[0,T]}(\cdot) \gamma(2^{k_1^*} \cdot -m) u_{k_1}](\tau,n) \text{ and } 1_{[0,T]}(\cdot) \gamma(2^{k_1^*} \cdot - m) \neq \gamma(2^{k_1^*} \cdot - m).
\end{equation*}
The additional decomposition in modulation is given by $f_{k_1} = \sum_{j \geq 0} f_{k_1,j}$.\\
Suppose below that $j_1=j_1^*$. Like above we find
\begin{equation*}
\begin{split}
&\quad \sum_{k,j_i} \frac{\tilde{a}(2^{k_1^*})}{2^{2k_1^*}} 2^k \Vert f_{k_1,j_1} \Vert_{L^2_{\tau} \ell^2_n} \Vert 1_{I_{k_1}} (f_{k_2,j_2}*f_{k_3,j_3}*f_{k_4,j_4})^{\mathfrak{N}} \Vert_{L^2_{\tau} \ell^2_n} \\
&\lesssim \sum_{k,j_i} \frac{\tilde{a}(2^{k_1^*})}{2^{2k_1^*}} 2^k 2^{j_1(1/2-\varepsilon)} 2^{-j_1(1/2-\varepsilon)} \Vert f_{k_1,j_1} \Vert_{L^2_{\tau} \ell^2_n} \Vert 1_{I_{k_1}} (f_{k_2,j_2}*f_{k_3,j_3}*f_{k_4,j_4})^{\mathfrak{N}} \Vert_{L^2_{\tau} \ell^2_n} \\
&\lesssim \frac{\tilde{a}(2^{k_1^*})}{2^{2k_1^*}} 2^{(3/2+\varepsilon)k_1^*} T^\theta \prod_{i=1}^4 \Vert u_i \Vert_{F^\alpha_{k_i}}
\end{split}
\end{equation*}
The ultimate estimate follows from Lemma \ref{lem:sharpTimeCutoffAlmostBounded}.

Also in the case $j_1 \neq j_1^*$, it is easy to see from Lemma \ref{lem:sharpTimeCutoffAlmostBounded} that the loss in highest modulation is more than compensated from the fact that $\# \mathcal{B} \leq 4$ independently of $k_1^*$.\\
We turn to $High \times High \times Low \times Low$-interaction. Suppose that $k_1 \leq k_2 \leq k_3-5$. From Lemma \ref{lem:pointwiseMultiplierBound} and introducing the dyadic sum $\sum_{\substack{2k_1^* \leq k \leq 2k_1^* + k_3^*,\\|\Omega| \sim 2^k}}$, we find
\begin{equation*}
\begin{split}
&\quad \sum_{\substack{2k_1^* \leq k \leq 2k_1^* + k_3^*+5,\\|\Omega| \sim 2^k}} \sum_{m \in \mathcal{A}} \frac{\tilde{a}(2^{k_1^*})}{2^{2k_1^*}} |\Omega|^{1/2} 2^{k_4^*/2} 2^{-k_1^*/2} \\
&\lesssim T 2^{k_1^*} \sum_{0 \leq k \leq k_3^*} \frac{\tilde{a}(2^{k_1^*})}{2^{k_1^*}} 2^{k/2} 2^{k_4^*/2} 2^{-k_1^*/2} \\
&\lesssim T \tilde{a}(2^{k_1^*}) 2^{k_3^*/2} 2^{k_4^*/2} 2^{-k_1^*/2} \lesssim T \prod_{i=1}^4 2^{(s-)k_i}
\end{split}
\end{equation*}
 provided that $s>1/4$.
 
Suppose that $j_1=j_1^*$. Together with the resonance identity $2^{j_1^*} \gtrsim |\Omega(\overline{n})| \sim 2^k$ and Lemma \ref{lem:RefinedStrichartzEstimateL6}, it follows
\begin{equation*}
\begin{split}
&\quad \int_{\tau_1+\ldots+\tau_4=0} \sum_{\substack{ n_1+\ldots+n_4=0,\\(*)}} \prod_{i=1}^4 f_{k_i,j_i}(\tau_i,n_i) \\
&\lesssim 2^{j_1/2} \Vert f_{k_1,j_1} \Vert_{L_\tau^2 \ell^2_n} 2^{-k/2} \Vert 1_{I_{k_1}} (f_{k_2,j_2}*f_{k_3,j_3}*f_{k_4,j_4})^{\mathfrak{N}} \Vert_{L_\tau^2 \ell_n^2} \\
&\lesssim 2^{-k/2} 2^{k_4^*/2} 2^{-k_1^*/2} \prod_{i=1}^4 2^{j_i/2} \Vert f_{k_i,j_i} \Vert_{L^2_{\tau} \ell^2_n}.
\end{split}
\end{equation*}
We estimate $High \times High \times High \times Low$-interaction. Suppose that $|k_1-k_2| \leq 10, |k_1-k_3| \leq 10$ and $k_4 \leq k_1-15$.\\
Lemma \ref{lem:pointwiseMultiplierBound} together with the magnitude of the resonance function $|\Omega| \sim 2^{3k_1^*}$ leads us to consider
\begin{equation*}
\begin{split}
\sum_{m \in \mathcal{A}} \frac{\tilde{a}(2^{k_1^*})}{2^{2k_1^*}} |\Omega|^{1/2} 2^{k_4^*/2} 2^{-k_1^*/2} &\lesssim T \frac{\tilde{a}(2^{k_1^*})}{2^{k_1^*}} |\Omega|^{1/2} 2^{k_4^*/2} 2^{-k_1^*/2} \lesssim T \tilde{a}(2^{k_1^*}) 2^{k_4^*/2}\\
&\lesssim T \prod_{i=1}^4 2^{(s-)k_i},
\end{split}
\end{equation*}
which proves the claim together with
\begin{equation*}
\begin{split}
&\quad \int_{\tau_1+\ldots+\tau_4=0} \sum_{\substack{ n_1+\ldots+n_4=0,\\(*)}} \prod_{i=1}^4 f_{k_i,j_i}(\tau_i,n_i) \\
&\lesssim 2^{j_1/2} \Vert f_{k_1,j_1} \Vert_{L_\tau^2 \ell_n^2} |\Omega(\overline{n})|^{-1/2} \Vert 1_{I_{k_1}} (f_{k_2,j_2}*f_{k_3,j_3}*f_{k_4,j_4})^{\mathfrak{N}} \Vert_{L_\tau^2 \ell_n^2} \\
&\lesssim |\Omega|^{-1/2} 2^{k_4^*/2} 2^{-k_1^*/2} \prod_{i=1}^4 2^{j_i/2} \Vert f_{k_i,j_i} \Vert_{L^2_{\tau} \ell^2_n}.
\end{split}
\end{equation*}
The proof is complete. 
\end{proof}
We prove the estimate for the boundary term:
\begin{lemma}
\label{lem:estimateBoundaryTerm}
Suppose that $s \in (-1/2,1/2)$ and $T \in (0,1]$. Then, there is $\varepsilon(s)>0$ and $d(s)>0$ such that the following estimate holds:
\begin{equation*}
B^{s,a,M}_4(T,u) \lesssim M^{-d(s)} \Vert u \Vert^4_{F^{s,\alpha}(T)}.
\end{equation*}
\end{lemma}
\begin{proof}
We localize frequencies on a dyadic scale, i.e., $P_{k_i} u_i = u_i$. Suppose by means of symmetry that $k_1 \geq k_2 \geq k_3 \geq k_4$ and set $m=\log_2(M)$. By virtue of Lemma \ref{lem:embeddingShorttimeSpaces}, it will be enough to derive a bound in terms of the Sobolev norms. We use a pointwise bound for $\psi_{s,a}$, which hinges on the sign of $s$. In the following we only consider $s>0$. It is straight-forward to check that the same argument yields the bound for $-1/2<s \leq 0$ using the bound from Remark \ref{rem:RegularityMultiplierNegativeSobolevSpacesMKDV} instead of the one due Lemma \ref{lem:estimateBoundaryTerm}.

For the evaluation at $t=0$, we have due to Lemma \ref{lem:pointwiseMultiplierBound} and an application of H\"older's inequality in position space
\begin{equation*}
\begin{split}
&\quad \frac{\tilde{a}(2^{k_1})}{2^{2k_1}} \sum_{\substack{n_1+\ldots+n_4=0,\\(*),|n_1| \geq M}} |\hat{u}_1(0,n_1)| |\hat{u}_2(0,n_2)| |\hat{u}_3(0,n_3)| |\hat{u}_4(0,n_4)|\\
&\lesssim \frac{\tilde{a}(2^{k_1})}{2^{2k_1}} \left( \Vert u_1(0) \Vert_{L_x^2} \Vert u_2(0) \Vert_{L_x^2} \Vert u^\prime_3(0) \Vert_{L_x^\infty} \Vert u^\prime_4(0) \Vert_{L_x^\infty} \right),
\end{split}
\end{equation*}
where $u_l^\prime(0,x) = \sum |\hat{u}_l(0,k)| e^{ikx}$, $l \in \{3,4\}$.

After applying Bernstein's inequality, we sum up the dyadic pieces
\begin{equation*}
\begin{split}
&\quad \sum_{\substack{k_1, k_2 \geq k_3 \geq k_4 \geq 0,\\(*),k_1 \geq m}} \frac{\tilde{a}(2^{k_1})}{2^{2k_1}} \Vert u_1(0) \Vert_{L_x^2} \Vert u_2(0) \Vert_{L_x^2} 2^{k_3/2} \Vert u_3(0) \Vert_{L_x^2} 2^{k_4/2} \Vert u_4(0) \Vert_{L_x^2} \\
&\lesssim \sum_{k_1, k_2 \geq k_3, k_1 \geq m} \frac{\tilde{a}(2^{k_1})}{2^{2k_1}} \Vert u_1(0) \Vert_{L_x^2} \Vert u_2(0) \Vert_{L_x^2} 2^{k_3/2} \Vert u_3(0) \Vert_{L_x^2} 2^{(1/2-s)k_3} \Vert u(0) \Vert_{H_x^s} \\
&\lesssim \sum_{k_1, k_2 \geq m} \frac{\tilde{a}(2^{k_1})}{2^{k_1}} 2^{-2sk_1} \Vert u_1(0) \Vert_{L_x^2} \Vert u_2(0) \Vert_{L_x^2} \Vert u(0) \Vert_{H_x^s}^2 \\
&\lesssim M^{-d(s)} \Vert u(0) \Vert_{H_x^{s}}^4,
\end{split}
\end{equation*}
which we can arrange as long as $s>-1/2$ choosing $\varepsilon = \varepsilon(s)$ sufficiently small. 
\end{proof}
Next, we derive the crucial bound for the correction term $R_6^{s,a,M}(T,u_1,\ldots,u_6)$. With the frequency constraint irrelevant here, we drop it in the following.
\begin{proposition}
\label{prop:EstimateCorrectionTerm}
Let $T \in (0,1]$. For $s>0$ and $\alpha = 1$, there is $\theta(s) > 0$ such that we find the following estimate to hold:
\begin{equation}
\label{eq:totalRemainderEstimateB}
R^{s,a,M}_6(T,u_1,\ldots,u_6) \lesssim T^\theta \prod_{i=1}^6 \Vert u_i \Vert_{F^{s,\alpha}(T)}
\end{equation}
\end{proposition}
\begin{proof}
We use the same reductions like in the proof of Proposition \ref{prop:symmetrizedEnergyEstimate}. Firstly, apply a decomposition into intervals in frequency space. We estimate $R^{s,a,M}_6(T,u_1,\ldots,u_6)$ for frequency localized functions $u_i$ satisfying $P_{k_i} u_i = u_i$. For these functions we show the estimate
\begin{equation}
\label{eq:localizedRemainderEstimateB}
R^{s,a,M}_6(T,u_1,\ldots,u_6) \lesssim T \prod_{i=1}^6 2^{(s-)k_i} \Vert u_i \Vert_{F^{\alpha}_{k_i}},
\end{equation}
and \eqref{eq:totalRemainderEstimateB} follows from \eqref{eq:localizedRemainderEstimateB} by the above arguments.

Estimate of $I(T)$:\\
Localize time antiproportionally to the highest frequency. Like above let $\gamma: \R \rightarrow [0,1]$ be a smooth function with compact support in $[-1,1]$ and the property
\begin{equation*}
\sum_{m \in \Z} \gamma^6(x-m) \equiv 1.
\end{equation*}
With this function write
\begin{equation*}
\begin{split}
I(T,u) &= \int_0^T dt \sum_{\substack{n_1+\ldots+n_4=0,\\(*)}} \frac{\psi_{s,a}(\overline{n}) n_1}{\Omega(\overline{n})} |\hat{u}(t,n_1)|^2 \hat{u}(t,n_1) \hat{u}(t,n_2) \hat{u}(t,n_3) \hat{u}(t,n_4) \\
&= \int_0^T dt \sum_{m \in \Z} \sum_{\substack{n_1+\ldots+n_4=0,\\(*)}} \frac{\psi_{s,a}(\overline{n}) n_1}{\Omega(\overline{n})} \gamma(2^{\alpha k_1^*} t -m) \hat{u}_1(t,n_1) \\
&\quad \quad \overline{\gamma(2^{\alpha k_1^*}t-m) \hat{u}_1(t,n_1)} \prod_{i=1}^4 \gamma(2^{\alpha k_1^*}t-m) \hat{u}_i(t,n_i)
\end{split}
\end{equation*}

This time we confine ourselves to the majority of the cases, where the smooth cutoff function does not interact with the sharp cutoff.\\
Denote 
\begin{equation*}
f_{k_{1a}}(\tau,n)=
\begin{cases}
\mathcal{F}_t[\gamma(2^{\alpha k_1^*}\cdot-m) \hat{u}_1(\cdot,n)], \quad a=1,3; \\
\mathcal{F}_t[\overline{\gamma(2^{\alpha k_1^*}\cdot-m) \hat{u}_1(\cdot,n)}], \quad a=2.
\end{cases}
\end{equation*}
and $f_{k_b}(\tau,n) = \mathcal{F}_t[\gamma(2^{\alpha k_1^*}\cdot-m) \hat{u}_j(\cdot,n)]$ for $j =2,3,4$.\\
Localization with respect to modulation is denoted by
\begin{equation*}
f_{k_i,j_i} = 
\begin{cases}
	\eta_{\leq j_i}(\tau - n^3) f_{k_i}, \quad j_i = [\alpha k_1^*]; \\
	\eta_{j_i}(\tau - n^3) f_{k_i}, \quad j_i > [\alpha k_1^*].
\end{cases}
\end{equation*}
By Lemma \ref{lem:pointwiseMultiplierBound}, the multiplier is estimated by
\begin{equation*}
\left| \frac{\psi_{s,a}(\overline{n})}{\Omega(\overline{n})} n_1 \right| \lesssim \frac{\tilde{a}(2^{k_1^*})}{2^{2k_1^*}} 2^{k_1}.
\end{equation*}
This leaves us with estimating the following expression, for which we assume the space-time Fourier transforms to be non-negative:
\begin{equation*}
\sum_{\substack{ j_{1a},j_b \\ a,b=1,2,3}} \sum_{\substack{n_1+\ldots+n_4=0,\\(*)}} \int_{\substack{ \tau_{11}+\tau_{12}+\tau_{13} \\ +\tau_2+\tau_3+\tau_4 = 0}} \prod_{i=1}^3 f_{k_{1i},j_{1i}}(\tau_{1i},n_1) \prod_{i=2}^4 f_{k_i,j_i}(\tau_i,n_i).
\end{equation*}

With the localization in modulation less relevant in the following, we do not write out the sum in the modulation variable. We apply the Cauchy-Schwarz inequality in the modulation variables and $n_1$ to find
\begin{equation}
\label{eq:RemainderReductionI}
\begin{split}
&\quad \sum_{n_1} \int_{\R} d\tau_1 \int_{\tau_1 = \tau_{11}+\tau_{12}+\tau_{13}} \prod_{i=1}^3 f_{k_{1i},j_{1i}}(\tau_{1i},n_1) \int_{\tau_1 + \ldots + \tau_4=0} \sum_{n_2,n_3, (*)} \prod_{i=2}^4 f_{k_i,j_i}(\tau_i,n_i) \\
&\lesssim \sum_{n_1} \left( \int_{\R} d\tau_1  \left| \int_{\tau_{1}=\tau_{11}+\tau_{12}+\tau_{13}} \prod_{i=1}^3 f_{k_{1i},j_{1i}} \right|^2 \right)^{1/2} \\
&\quad \quad \left( \int_{\R} d\tau_1 \left| \int_{\tau_1+\ldots+\tau_4 = 0} \sum_{\substack{n_2,n_3,\\(*)}} \prod_{i=2}^4 f_{k_i,j_i}(\tau_i,n_i) \right|^2 \right)^{1/2}
\end{split}
\end{equation}
For the first factor we find by two applications of the Cauchy-Schwarz inequality
\begin{equation*}
\left( \int d\tau_1 \left| \int_{\tau_{11}+\tau_{12}+\tau_{13}=\tau_1} \prod_{i=1}^3 f_{k_{1i},j_{1i}}(\tau_{1i},n_1) \right|^2 \right)^{1/2} \lesssim 2^{j_{11}/2} 2^{j_{12}/2} \prod_{i=1}^3 \Vert f_{k_{1i},j_{1i}} \Vert_{L^2_{\tau}} 
\end{equation*}
Next, an application of H\"older's inequality in $n_1$ yields
\begin{equation*}
\eqref{eq:RemainderReductionI} \lesssim 2^{-k_1^*/2} \prod_{i=1}^3 2^{j_i/2} \left( \sum_{n_1} \Vert f_{k_{1i},j_{1i}}(\cdot,n_1)\Vert_{L^2_{\tau}}^6 \right)^{1/3} \Vert 1_{I_{k_1}} (f_{2} * f_3 * f_4)^{\mathfrak{N}} \Vert_{L^2_{\tau} \ell^2_n} 
\end{equation*}

Applications of the embedding $\ell^2 \hookrightarrow \ell^6$, Young's inequality and Lemma \ref{lem:RefinedStrichartzEstimateL6} yield
\begin{equation*}
\lesssim 2^{-k_1^*/2} \prod_{i=1}^3 2^{j_{1i}/2} \Vert f_{k_{1i},j_{1i}} \Vert_{L^2_{\tau} \ell^2_n} \prod_{i=2}^4 2^{j_i/2} \Vert f_{k_i,j_i} \Vert_{L^2_{\tau} \ell^2_n}
\end{equation*}
Gathering all factors, we have derived the estimate
\begin{equation*}
\begin{split}
I(T,u_1,\ldots,u_4) &\lesssim T \sum_{k_1 \leq k_1^*} \frac{\tilde{a}(2^{k_1^*})}{2^{2k_1^*}} 2^{k_1} 2^{k_1^*} 2^{-k_1^*/2} \\
 &\quad \prod_{i=1}^3 \sum_{j_{1i}\geq k_1^*}  2^{j_{1i}/2} \Vert f_{k_{1i},j_{1i}} \Vert_{L^2_{\tau} \ell^2_n} \prod_{i=2}^4 \sum_{j_i \geq k_1^*} 2^{j_i/2} \Vert f_{k_i,j_i} \Vert_{L^2_{\tau} \ell^2_n},
\end{split}
\end{equation*}
and estimate \eqref{eq:localizedRemainderEstimateB} follows even for negative $s$ due to \eqref{eq:XkEstimateIII}.\\
\begin{equation*}
\begin{split}
\text{Estimate of } II(T,u_1,\ldots,u_6) &= \int_0^T dt \sum_{\substack{n_1+\ldots+n_4 = 0,\\ (*)}} \frac{\psi_{s,a}(\overline{n}) n_1}{\Omega(\overline{n})} \prod_{i=2}^4 \hat{u}(t,n_i)  \\
 &\quad \quad \sum_{\substack{n_1+n_5+n_6+n_7 = 0,\\(*)}} \prod_{i=5}^7 \hat{u}(t,n_i):
\end{split}
\end{equation*}

With the notation and the reductions from above, we will show that
\begin{equation*}
II(T,u_1,\ldots,u_6) \lesssim T^\theta \prod_{i=1}^6 2^{(s-)k_i} \Vert u_i \Vert_{F^{\alpha}_{k_i}}.
\end{equation*}
Also, we use an additional dyadic decomposition for $n_1$. We assume in the following $n_1 \in I_{k_1}$ and additionally sum over $k_1$.
We denote the decreasing arrangements of $k_2,k_3,k_4$ by $a_1^*,a_2^*,a_3^*$ and of $k_5,k_6,k_7$ by $b_1^*,b_2^*,b_3^*$ and note that $k_1 \leq a_1^*+5$ due to impossible frequency interaction. The cases $k_1^* = a_1^*$ and $k_1^* = b_1^*$ are analyzed separately.

\textbf{Case A:} $k_1^* = a_1^*$:\\
We localize time according to $k_1^*$. Lemma \ref{lem:pointwiseMultiplierBound} gives the estimate
\begin{equation*}
\left| \frac{|\psi_{s,a}(\overline{n})|}{|\Omega(\overline{n})|} n_1 \right| \lesssim \frac{\tilde{a}(2^{k_1^*})}{2^{2k_1^*}} 2^{k_1}.
\end{equation*}
Next, the above reductions are applied. Introduce an additional partition in the modulation variables (although the sum is not written out anymore) and apply the triangle inequality to arrive at the following expression: 
\begin{equation*}
\begin{split}
&\sum_{0 \leq k_1 \leq k_1^* + 5} 2^{k_1} \frac{\tilde{a}(2^{k_1^*})}{2^{2k_1^*}} \int_{\tau_2+\ldots+\tau_7=0} \sum_{\substack{n_1+\ldots+n_4=0,\\(*),n_1 \in I_{k_1}}} \prod_{i=2}^4 f_{k_i,j_i}(\tau_i,n_i) \\
&\quad \quad \sum_{\substack{n_1+n_5+n_6+n_7=0,\\(*),n_1 \in I_{k_1}}} \prod_{i=5}^7 f_{k_i,j_i}(\tau_i,n_i) \\
&\lesssim \sum_{0 \leq k_1 \leq k_1^*+5} 2^{k_1} \frac{\tilde{a}(2^{k_1^*})}{2^{2k_1^*}} \Vert 1_{I_{k_1}} (f_{k_2,j_2}*f_{k_3,j_3}*f_{k_4,j_4})^{\mathfrak{N}} \Vert_{L^2_{\tau} \ell^2_n} \\
&\quad \quad \Vert 1_{I_{k_1}} (f_{k_5,j_5} * f_{k_6,j_6} * f_{k_7,j_7})^{\mathfrak{N}} \Vert_{L^2_{\tau} \ell^2_n}
\end{split}
\end{equation*}
The refined $L_{t,x}^6$-Strichartz estimate from Lemma \ref{lem:RefinedStrichartzEstimateL6} is applied twice to find
\begin{equation*}
\begin{split}
&\lesssim \sum_{1 \leq k_1 \leq a_1^*} 2^{k_1} \frac{\tilde{a}(2^{k_1^*})}{2^{2k_1^*}} 2^{-a_1^*/2} 2^{a_3^*/2} \prod_{i=2}^4 2^{j_i/2} \Vert f_{k_i,j_i} \Vert_{L^2_{\tau} \ell^2_n} \\
&\quad \quad 2^{-b_1^*/2} 2^{b_3^*/2} \prod_{i=5}^7 2^{j_i/2} \Vert f_{k_i,j_i} \Vert_{L^2_{\tau} \ell^2_n}\\
&\lesssim \frac{\tilde{a}(2^{k_1^*})}{2^{k_1^*}} 2^{-k_1^*/2} 2^{a_3^*/2} 2^{-b_1^*/2} 2^{b_3^*/2} \prod_{i=2}^7 2^{j_i/2} \Vert f_{i,j_i} \Vert_{L^2_{\tau} \ell^2_n}.
\end{split}
\end{equation*}
Taking the localization in time into account, which gives an additional factor of $T 2^{k_1^*}$ in the majority of the cases, we find \eqref{eq:localizedRemainderEstimateB} to hold for $s>0$ after summing over $j_i$ and invoking \eqref{eq:XkEstimateIII}.

\textbf{Case B:} $k_1^* = b_1^*$:\\
We localize time only according to $a_1^*$. This gives a factor of $T 2^{a_1^*}$. Below, we denote the localized functions $\hat{u}(t,n) \gamma(2^{\alpha a_1^*}t-m)$ again by $\hat{u}(t,n)$. We estimate the remaining expression as follows:
\begin{equation}
\label{eq:IntermediateStepCorrectionTermEstimate}
\begin{split}
&\quad \sum_{0 \leq k_1 \leq a_1^*+5} \int_{\R} dt \sum_{\substack{n_1+\ldots+n_4=0,\\(*)}} \frac{\psi_{s,a}(\overline{n})}{\Omega(\overline{n})} n_1 1_{I_{k_1}}(n_1) \prod_{i=2}^4 \hat{u}_i(t,n_i) \\
&\quad \quad \sum_{\substack{n_1 + n_5+ n_6+n_7=0,\\(*)}} \prod_{i=5}^7 \hat{u}_5(t,n_i) \\
&\lesssim \sum_{0 \leq k_1 \leq a_1^*+5} 2^{k_1} \int_{\R} dt \sum_{n_1 \in I_{k_1}} | \sum_{\substack{n_1+n_5+n_6+n_7=0,\\(*)}} \prod_{i=5}^7 \hat{u}_i(t,n_i)  | \\
&\quad \quad | \sum_{n_2,n_3,(*)} \frac{\psi_{s,a}(\overline{n})}{\Omega(\overline{n})} \hat{u}_2(t,n_2) \hat{u}_3(t,n_3) \hat{u}_4(t,-n_1-n_2-n_3) | \\
&\lesssim \sum_{0 \leq k_1 \leq a_1^*+5} 2^{k_1} \int_{\R} dt ( \sum_{n_1 \in I_{k_1}} | \sum_{\substack{n_1+n_5+n_6+n_7=0,\\(*)}} \prod_{i=5}^7 \hat{u}_i(t,n_i) |^2 )^{1/2} \\
&\quad \quad ( \sum_{n_1 \in I_{k_1}} | \sum_{n_2,n_3,(*)} \frac{\psi_{s,a}(\overline{n})}{\Omega(\overline{n})} \hat{u}_2(t,n_2) \hat{u}_3(t,n_3) \hat{u}_4(t,-n_1-n_2-n_3) |^2 )^{1/2} \\
\end{split}
\end{equation}
Next, we apply H\"older's inequality in time, and for $\hat{u}_2$, $\hat{u}_3$ and $\hat{u}_4$ we already plug-in the decomposition in the modulation variable adapted to the localization in time. We start with a size of the modulation variable of $2^{a_1^*}$. Further, we assume $f_{k_i,j_i} \geq 0$. We find from applying Plancherel's theorem and the refined Strichartz estimate
\begin{equation*}
\begin{split}
&\quad \sum_{j_2,j_3,j_4 \geq a_1^*} \left\Vert \left( \sum_{n_1 \in I_{k_1}} \left| \sum_{n_2,n_3,(*)} \frac{\psi_{s,a}(\overline{n})}{\Omega(\overline{n})} \hat{u}_{2,j_2}(t,n_2) \hat{u}_{3,j_3}(t,n_3) \hat{u}_{4,j_4}(t,n_4) \right|^2 \right)^{1/2} \right\Vert_{L_t^2} \\
&\lesssim \frac{\tilde{a}(2^{a_1^*})}{2^{2a_1^*}} \sum_{j_2,j_3,j_4 \geq a_1^*} \Vert 1_{I_{k_1}} (f_{k_2,j_2} * f_{k_3,j_3} * f_{k_4,j_4})^{\mathfrak{N}} \Vert_{L^2_{\tau} \ell^2_n} \\
&\lesssim \frac{\tilde{a}(2^{a_1^*})}{2^{2a_1^*}} 2^{k_1/2} 2^{-a_1^*/2} \prod_{i=2}^4 \sum_{j_i \geq a_1^*} 2^{j_i/2} \Vert f_{k_i,j_i} \Vert_{L^2_{\tau} \ell^2_n}.
\end{split}
\end{equation*}

We note that for the other convolution term in \eqref{eq:IntermediateStepCorrectionTermEstimate} the localization in time is not high enough to finally evaluate the factors in $F^1_{k_i}$. Thus, we increase localization in time to $2^{k_1^*}$. To derive more favourable bounds, we use orthogonality in time.\\
Observe that for $f: \R \rightarrow \C$ with $\text{supp}_t(f) \subseteq I$, $I$ an interval with $|I|=2^{-a}$ and $J \subseteq I$ a family of intervals partitioning $I$ with $|J|=2^{-b}$, $a<b$:
\begin{equation*}
\Vert f(t) \Vert_{L_t^2(I)}^2 = \Vert \sum_{\substack{J \subseteq I, \\ |J|=2^{-b}}} 1_{J}(t) f(t) \Vert_{L_t^2}^2 = \sum_{J \subseteq I} \Vert 1_J(t) f(t) \Vert_{L_t^2}^2 \lesssim 2^{(b-a)/2} \sup_{J \subseteq I} \Vert f \Vert^2_{L^2_t(J)}
\end{equation*}
In the present context, due to the time localization up to $2^{-a_1^*}$, which was already given, increasing localization in time to $2^{-k_1^*}$ only amounts to a factor $2^{k_1^*/2} 2^{-a_1^*/2}$.\\
Further, we localize in modulation and suppose $f_{k_i,j_i} \geq 0$. Using Plancherel's theorem and the refined Strichartz estimate, we conclude the bound
\begin{equation*}
\begin{split}
&\quad \sum_{j_5,j_6,j_7 \geq b_1^*} \left\Vert \left( \sum_{n_1} \left| \sum_{n_5,n_6,(*)} \prod_{i=5}^7 \hat{u}_i(t,n_i)  \right|^2 \right)^{1/2} \right\Vert_{L_t^2} \\
&\lesssim \sum_{j_5,j_6,j_7 \geq b_1^*} \Vert 1_{I_{k_1}} (f_{k_5,j_5} * f_{k_6,j_6} * f_{k_7,j_7})^{\mathfrak{N}} \Vert_{L^2_{\tau} \ell^2_n} \\
&\lesssim 2^{-k_1^*/2} 2^{k_1/2} \prod_{i=5}^7 \sum_{j_i \geq b_1^*} 2^{j_i/2} \Vert f_{k_i,j_i} \Vert_{L^2_{\tau} \ell^2_n}.
\end{split}
\end{equation*}
We gather all factors to find
\begin{equation*}
\begin{split}
&\quad T \sum_{1 \leq k_1 \leq a_1^*} 2^{k_1} \frac{\tilde{a}(2^{a_1^*})}{2^{2a_1^*}} 2^{a_1^*} 2^{b_1^*/2} 2^{-a_1^*/2} 2^{-b_1^*/2} 2^{k_1/2} \prod_{i=2}^7 2^{j_i/2} \Vert f_{k_i,j_i} \Vert_{L^2_{\tau} \ell^2_n} \\
&\lesssim T \tilde{a}(2^{a_1^*}) \prod_{i=2}^7 2^{j_i/2} \Vert f_{i,j_i} \Vert_{L^2_{\tau} \ell^2_n}.
\end{split}
\end{equation*}
We find \eqref{eq:localizedRemainderEstimateB} to hold for $s>0$ and $\varepsilon(s)>0$ due to \eqref{eq:XkEstimateIII}.

The proof is complete. 
\end{proof}
\subsection{Conditional energy estimates in negative Sobolev spaces}
\label{subsection:EnergyEstimatesNegativeSobolevSpaces}
Next, we see how under the hypothesis \eqref{eq:L8Hypothesis} of an essentially sharp $L^8_{t,x}$-Strichartz estimate energy estimates in negative Sobolev spaces for functions in $F^{1+\delta}$-spaces for some $\delta > 0$ can be shown. Recall from the beginning of Section \ref{section:MultilinearEstimates} how via interpolation follows
\begin{equation}
\label{eq:L6StrichartzHypothesis}
\Vert u \Vert_{L^6_{t,x}(\R \times \T)} \lesssim \Vert u \Vert_{X^{0+,4/9+}}.
\end{equation}

In the proof of energy estimates in negative Sobolev spaces, the smoothing of \eqref{eq:L6StrichartzHypothesis} in short-time spaces will be utilized. We stress that it is estimate \eqref{eq:L6StrichartzHypothesis}, which is required to prove Proposition \ref{prop:negativeEnergyPropagationMKDV}. Theorem \ref{thm:nonExistence} is formulated conditional upon the essentially sharp $L^8_{t,x}$-Strichartz estimate \eqref{eq:L8Hypothesis} as this conjecture is more prominent. However, we suspect that \eqref{eq:L6StrichartzHypothesis} is easier to prove than \eqref{eq:L8Hypothesis}.

Further, in Proposition \ref{prop:negativeEnergyPropagationMKDV} we shall only prove a qualitative result as \eqref{eq:L6StrichartzHypothesis} is currently out of reach. The analysis of Subsection \ref{subsection:EnergyEstimatesPositiveSobolevSpaces} implies favourable bounds in negative Sobolev spaces for all terms, but $R_6^{s,a}$. This contribution is analyzed in the remainder of this subsection.
\begin{proposition}
\label{prop:negativeEnergyPropagationMKDV}
Let $T \in (0,1]$ and suppose that \eqref{eq:L6StrichartzHypothesis} is true. There is $\delta^\prime > 0$ and $\theta > 0$ so that for $0<\delta<\delta^{\prime}$ there is $s=s(\delta)<0$ such that the following estimate holds:
\begin{equation}
R_6^{s,a,M} \lesssim T^\theta \prod_{i=1}^6 \Vert u_i \Vert_{F^{s,1+\delta}(T)}
\end{equation}
\end{proposition}
\begin{proof}
In the proof of Proposition \ref{prop:energyPropagation} was shown that $I(T,u)$ can be estimated in negative Sobolev spaces. Thus, we only estimate $II(T,u)$ below.

As above the frequency constraint is omitted, and $R_6^{s,a,M}$ is split into dyadic blocks $R_6^{s,a,M}(K_2,\ldots,K_7)$ where $\text{supp} \; \hat{u}_i \subseteq I_{k_i}$, $K_i = 2^{k_i}$. We may assume by symmetry that $K_2 \geq K_3 \geq K_4$, $K_5 \geq K_6 \geq K_7$. Further, let $K_1^* \geq \ldots \geq K_6^*$ denote a decreasing rearrangement of $K_i, \; i=2,\ldots,7$.

\textbf{Case A:} $K_2 \gtrsim K_5$. In this case $K_1^* \sim K_2$, and we add localization in time according to $K_2$. Let $\gamma$ be a smooth function with support in $[-1,1]$ satisfying
\begin{equation*}
\sum_m \gamma^6(t-m) \equiv 1.
\end{equation*}
We have to estimate
\begin{equation*}
\begin{split}
&\sum_m \int_{\R} dt 1_{[0,T](t)} \sum_{\substack{n_1+n_2+n_3+n_4=0,\\(*)}} \frac{\psi_{s,a}(\overline{n}) n_1}{\Omega(\overline{n})} \prod_{i=2}^4 \gamma(2^{(1+\delta) k_1^*} t -m) \hat{u}_i(t,n_i) \\
&\quad \times \sum_{\substack{n_1+n_5+n_6+n_7=0,\\(*)}} \prod_{i=5}^7 \gamma(2^{(1+\delta) k_1^*} t -m) \hat{u}_i(t,n_i).
\end{split}
\end{equation*}
First, we handle the majority of cases, for which
\begin{equation*}
1_{[0,T]}(\cdot) \gamma(2^{(1+\delta) k_1^*} \cdot - m) = \gamma(2^{(1+\delta) k_1^*} \cdot - m).
\end{equation*}
Let $f_{k_i} = \mathcal{F}_{t,x}[\gamma(2^{(1+\delta)k_1^*} \cdot - m) u_i]$.\\
This is further decomposed as $f_{k_i} = \sum_{j_i \geq (1+\delta) k_1^*} f_{k_i,j_i}$.\\
By the above, we have to estimate
\begin{equation*}
\sum_{k_1 \leq k_2} \frac{2^{k_1}}{2^{2k_1^*}} \Vert f_{k_2,j_2} * f_{k_3,j_3} * f_{k_4,j_4} \Vert_{L^2_{\tau} \ell^2_n} \Vert f_{k_5,j_5} * f_{k_6,j_6} * f_{k_7,j_7} \Vert_{L^2_{\tau}, \ell^2_n},
\end{equation*}
after which it remains to sum over $j_i \geq (1+\delta) k_1^*$ and take into account time localization, which amounts to a factor $T2^{(1+\delta)k_1^*}$.\\
Above, $a \in S^s_{\varepsilon}$ for negative $s$ is crudely bounded by a constant.

\eqref{eq:L6StrichartzHypothesis} yields for one factor
\begin{equation*}
\begin{split}
\Vert f_{k_2,j_2} * f_{k_3,j_3} * f_{k_4,j_4} \Vert_{L^2_{\tau} \ell^2_n} &\lesssim \Vert \mathcal{F}^{-1}_{t,x}[f_{k_2,j_2}] \cdot \mathcal{F}^{-1}_{t,x}[f_{k_3,j_3}] \cdot \mathcal{F}^{-1}_{t,x}[f_{k_4,j_4}] \Vert_{L^2_t L^2_x} \\
&\lesssim \Vert \mathcal{F}_{t,x}^{-1} f_{k_2,j_2} \Vert_{L^6_{t,x}} \Vert \mathcal{F}_{t,x}^{-1} f_{k_3,j_3} \Vert_{L^6_{t,x}} \Vert \mathcal{F}_{t,x}^{-1} f_{k_4,j_4} \Vert_{L^6_{t,x}} \\
&\lesssim 2^{(0k_1)+} \prod_{i=2}^4 2^{(4j_i/9)+} \Vert f_{k_i,j_i} \Vert_{2} 
\end{split}
\end{equation*}
and by \eqref{eq:XkEstimateIII}, we find the contribution of the majority of the cases to be bounded by
\begin{equation*}
\lesssim T 2^{(\delta k_1^*)+} \prod_{i=2}^7 2^{-(k_i/18)+} \Vert P_{k_i} u \Vert_{F_{k_i}^{1+\delta}}
\end{equation*}
with easy summation in certain negative Sobolev spaces.

%
\textbf{Case B:} $K_5 \sim K_6 \gg K_2$.\\
\underline{Subcase BI:} $K_5^2 \gg K_2^3$. Let
\begin{equation*}
\Omega^{(1)}(n_1,\ldots,n_4) = n_1^3 + n_2^3 + n_3^3 + n_4^3 \quad n_1+\ldots+n_4=0
\end{equation*}
denote the first resonance function, and
\begin{equation*}
\Omega^{(2)}(n_1,\ldots,n_6) = \sum_{i=1}^6 n_i^3, \quad n_1 + \ldots + n_6 = 0
\end{equation*}
denote the second resonance function.

In case $K_5^2 \gg K_2^3$, we find
\begin{equation*}
|\Omega^{(1)}(n_1,n_2,n_3,n_4)| \ll |\Omega^{(1)}(n_1,n_5,n_6,n_7)|
\end{equation*} 
and consequently, the second resonance function for the collected frequencies
\begin{equation*}
\Omega^{(2)}(n_2,n_3,n_4,-n_5,-n_6,-n_7) = \Omega^{(1)}(n_1,n_2,n_3,n_4) - \Omega^{(1)}(n_1,n_5,n_6,n_7)
\end{equation*}
satisfies $|\Omega^{(2)}| \sim |\Omega^{(1)}(n_1,n_5,n_6,n_7)| \gtrsim K_5^2$.\\
Let $\gamma$ be like in Case A. We add localization in time according to $K_5^{(1+\delta)}$, which leads us to estimate
\begin{equation*}
\begin{split}
&\int_{\R} dt \sum_{\substack{n_1+n_2+n_3+n_4=0,\\(*)}} \frac{\psi_{s,a}(\overline{n})}{\Omega(\overline{n})} n_1 \prod_{i=2}^4 \gamma(2^{(1+\delta)k_1^*} t-m) \hat{u}_2(t,n_i) \\
&\quad \quad \sum_{\substack{n_1+n_5+n_6+n_7 = 0,\\(*)}} 1_{[0,T]}(t) \prod_{i=5}^7 \gamma(2^{(1+\delta)k_1^*} t-m) \hat{u}_i(t,n_i).
\end{split}
\end{equation*}
We only deal with the majority of the cases, where
\begin{equation*}
1_{[0,T]}(\cdot) \gamma(2^{(1+\delta)k_1^*} \cdot - m) = \gamma(2^{(1+\delta)k_1^*} \cdot - m).
\end{equation*}
The systematic modification for the exceptional cases is omitted.

The idea is to use two bilinear Strichartz estimates from Lemma \ref{lem:bilinearEstimate} involving $u_5$, $u_6$, $u_7$ and the function with high modulation $j_i \geq 2k_5-10$. Suppose e.g. that $j_4 \geq 2k_5-10$.\\
Up to time localization factor and summation over $j_i \geq (1+\delta) k_1^*$, we find
\begin{equation*}
\begin{split}
&\quad \sum_{k_1 \leq k_2} \frac{2^{k_1}}{2^{2k_2^*}} \int (f_{k_2,j_2} * f_{k_3,j_3} * f_{k_4,j_4})(f_{k_5,j_5} * f_{k_6,j_6} * f_{k_7,j_7}) dt \\
&\lesssim \sum_{k_1 \leq k_2} \frac{2^{k_1}}{2^{2k_2^*}} \int u_{k_2,j_2} \ldots u_{k_7,j_7} dx dt,
\end{split}
\end{equation*}
where $u_{k_i,j_i} = \mathcal{F}_{t,x}^{-1}[f_{k_i,j_i}]$.\\
Here, we ignore the (in this case) irrelevant reflection $\tilde{f}(\tau,n) = f(-\tau,-n)$.

Thus, the majority of the cases is estimated by
\begin{equation*}
\lesssim T 2^{(\delta-1/4)k_5} \frac{2^{k_3/2}}{2^{k_2/2}} \prod_{i=2}^7 2^{j_i/2} \Vert f_{k_i,j_i} \Vert_2.
\end{equation*}
Hence, summation in negative Sobolev spaces is straight-forward for $\delta < 1/4$.

\underline{Subcase BII:} $K_5^2 \lesssim K_2^3$. In case $K_5 \sim K_7$ we find $|\Omega^{(1)}(n_1,n_5,n_6,n_7)| \sim K_5^3$ and consequently, $|\Omega^{(2)}| \sim K_5^3$. The argument from Subcase BI provides a sufficient estimate. Thus, suppose in the following $K_7 \ll K_5$.

\underline{Subsubcase BIIa:} $K_3 \ll K_2$. It has to hold $K_1 \sim K_2$.\\
If $K_2 \ll K_7$, then $|\Omega^{(1)}(n_1,n_5,n_6,n_7)| \gtrsim K_5^2 K_7 \gg |\Omega^{(1)}(n_1,n_2,n_3,n_4)|$.\\
If $K_7 \ll K_2$, then $|\Omega^{(1)}(n_1,n_5,n_6,n_7)| \gtrsim K_5^2 K_2 \gg |\Omega^{(1)}(n_1,n_2,n_3,n_4)|$ because $|\Omega^{(1)}(n_1,n_2,n_3,n_4)| \lesssim K_2^2 K_3$.\\
In any case, $|\Omega^{(2)}| \gtrsim K_5^2$ and the argument from Subcase BI is sufficient.

It remains to check $K_2 \sim K_7$. We separate variables via expansion into a rapidly converging Fourier series (the required regularity of the multiplier is provided following Remark \ref{rem:RegularityMultiplierNegativeSobolevSpacesMKDV} after Lemma \ref{lem:pointwiseMultiplierBound}). For details on this argument, see \cite[Section~5]{ChristHolmerTataru2012}. This leads to the expression
\begin{equation*}
\sum_{k_1 \leq k_2} \frac{2^{k_1}}{2^{2k_2}} \int_0^T dt \int dx u_{k_2} \ldots u_{k_7}
\end{equation*}
Let $\gamma$ be like above and by H\"older's inequality
\begin{equation*}
\begin{split}
&\quad \sum_{k_1 \leq k_2} \frac{2^{k_1}}{2^{2k_2}} \sum_{|m| \lesssim T 2^{(1+\delta)k_2}} \int_{\R} dt \int dx u_{k_2} \gamma(2^{(1+\delta)k_2}t -m) \ldots \gamma(2^{(1+\delta)k_2} t -m) \\
&\quad 1_{[0,T]}(t) u_{k_5} \gamma(\ldots) u_{k_6} \gamma(\ldots) u_{k_7} \gamma(\ldots) \\
&\lesssim 2^{-k_2} \sum_{|m| \lesssim T 2^{(1+\delta)k_2}} \Vert \gamma(2^{(1+\delta)k_2} t -m) u_{k_2} \Vert_{L^6_{t,x}} \ldots \\
&\quad \Vert u_{k_5} \gamma(\ldots) u_{k_6} \gamma(\ldots) 1_{[0,T]}(\cdot) \Vert_{L^3_{t,x}} \Vert u_{k_7} \gamma(\ldots) \Vert_{L^6_{t,x}}.
\end{split}
\end{equation*}

Decompose for $i \in \{2,3,4,7\}$
\begin{equation*}
f_{k_i} = \mathcal{F}_{t,x}[\gamma(2^{(1+\delta)k_2} t -m) u_{k_i}] = \sum_{j_i \geq (1+\delta)k_2} f_{k_i,j_i}
\end{equation*}
and by \eqref{eq:L6StrichartzHypothesis}
\begin{equation*}
\Vert \mathcal{F}_{t,x}^{-1}[f_{k_i}] \Vert_{L^6_{t,x}} \lesssim \sum_{j_i \geq (1+\delta)k_2} \Vert \mathcal{F}_{t,x}^{-1}[f_{k_i,j_i}] \Vert_{L^6_{t,x}} \lesssim \sum_{j_i \geq (1+\delta) k_2} 2^{(4j_i/9)+} \Vert f_{k_i,j_i} \Vert_{L^2}
\end{equation*}
for $i \in \{2,3,4,7\}$. For these functions time is localized sufficiently.

For the high frequencies we have to add localization in time, where we exploit orthogonality in time
\begin{equation*}
\begin{split}
&\quad \Vert u_{k_5} \gamma^2(2^{(1+\delta)k_2} t -m) u_{k_6} 1_{[0,T]} \Vert_{L_{t,x}^3} \\
&\lesssim \left( \sum_n \Vert u_{k_5} \gamma(2^{(1+\delta) k_2} t -m) \tilde{\gamma}(2^{(1+\delta)k_5} t -n) \right. \\
&\quad \quad \left. u_{k_6} \gamma(2^{(1+\delta)k_2} t -m) \tilde{\gamma}(2^{(1+\delta) k_5} t -n) \Vert_{L^3_{t,x}}^3 \right)^{1/3}.
\end{split}
\end{equation*}

Consequently, it is enough to estimate
\begin{equation*}
\left( 2^{(1+\delta)k_5}/2^{(1+\delta)k_2} \right)^{1/3} \Vert u_{k_5} \tilde{\gamma}(2^{(1+\delta) k_5} t -n) \Vert_{L^6_{t,x}} \Vert u_{k_6} \tilde{\gamma}(2^{(1+\delta) k_5} t -n) \Vert_{L^6_{t,x}},
\end{equation*}
which, by $k_5 \leq (3/2) k_2$, \eqref{eq:L6StrichartzHypothesis} and the above argument of splitting the modulation is achieved by
\begin{equation*}
\lesssim 2^{(0k_2)+} 2^{(1+\delta)k_2/6} \prod_{i=5}^6 \sum_{j_i \geq (1+\delta)k_5} 2^{(4j_i/9)+} \Vert f_{k_i,j_i} \Vert_2
\end{equation*}
Gathering all factors and invoking \eqref{eq:XkEstimateIII}, we have derived the bound
\begin{equation*}
R^6_{s,a}(K_2,\ldots,K_7) \lesssim T \frac{2^{k_2}}{2^{2k_2}} 2^{(1+\delta)k_2} 2^{(1+\delta)k_2/6} \prod_{i=2}^7 2^{-(k_i/18)+} \Vert u_i \Vert_{F_{k_i}^{1+\delta}}.
\end{equation*}
Since there are four factors with frequency higher or equal to $K_2$, there is enough smoothing from \eqref{eq:L6StrichartzHypothesis} to sum the expression even for negative regularities choosing $\delta$ sufficiently small.

\underline{Subsubcase BIIb:} $K_2 \sim K_3$.\\
If $K_7 \sim K_6$, then $|\Omega^{(1)}(n_1,n_5,n_6,n_7)| \gtrsim K_5^3 \gg K_2^3$ and the argument from Subcase BI applies.\\
Similarly, if $K_2 \ll K_7 \ll K_5$, then we find
\begin{equation*}
|\Omega^{(1)}(n_1,n_5,n_6,n_7)| \sim K_5^2 K_7 \gg K_2^3 \gtrsim |\Omega^{(1)}(n_1,n_2,n_3,n_4)|.
\end{equation*}
Thus, we can suppose that $K_7 \lesssim K_2$. In this case the argument from Subsubcase BIIa applies because there are at least two frequencies comparable to $K_2$ and at most two frequencies,  namely $K_5$ and $K_6$, much higher than $K_2$.

The proof is complete. 
\end{proof}
\begin{remark}
\label{rem:ExtraSmoothing}
We observe from the proofs of Propositions \ref{prop:symmetrizedEnergyEstimate}, \ref{prop:EstimateCorrectionTerm} and \ref{prop:negativeEnergyPropagationMKDV} and Lemma \ref{lem:estimateBoundaryTerm} that there is some slack in the regularity. In fact, we can lower the regularity on the right-hand side depending on $s$ (after making $\varepsilon = \varepsilon(s)$ smaller, if necessary). This observation becomes important in the construction of the data-to-solution mapping.
\end{remark}

\subsection{Conclusion of Proposition \ref{prop:energyPropagation}}
\label{subsection:ConclusionEnergyEstimate}
To conclude the proof of the energy estimate, we derive a bound for the thresholds of the frequency localized energy. We have the following lemma on frequency localized energy thresholds. Although in \cite{KochTataru2007} this lemma was only proved in the real line case, the proof for the torus carries over almost verbatim.
\begin{lemma}{\cite[Lemma~5.5.,~p.~34]{KochTataru2007}}
\label{lem:energyThreshold}
For any $u_0 \in H^s(\mathbb{T}) $ and $\varepsilon >0$ there is a sequence $(\beta_n)_{n \in \mathbb{N}_0}$ satisfying the following conditions:
\begin{enumerate}
\item[(a)] $2^{2ns} \Vert P_n u_0 \Vert_{L^2}^2 \leq \beta_n \Vert u_0 \Vert_{H^s}^2 $,
\item[(b)] $\sum_n \beta_n \lesssim 1$,
\item[(c)] $(\beta_n)$ satisfies a log-Lipschitz condition, that is
\begin{equation*}
| \log_2 \beta_{n} - \log_2 \beta_{m} | \leq \frac{\varepsilon}{2} |n - m|.
\end{equation*}
\end{enumerate}
\end{lemma}
To finish the proof of Proposition \ref{prop:energyPropagation}, choose an envelope sequence for initial data from the above lemma and define associated symbols. For details we refer to \cite{KochTataru2007}.

\section{Extension to related models}
\label{section:ExtensionFurtherModels}
Here, we illustrate how the analysis extends to related models. We remark that the energy method yields local solutions in $H^s$ for $s>3/2$ for \eqref{eq:mKdVmKdVSystem}, and for \eqref{eq:KdVmKdVEquation}, the analysis from \cite{Bourgain1993FourierTransformRestrictionPhenomenaII} applies and yields local well-posedness for $s \geq 1/2$.
\subsection{KdV-mKdV-equation}
With the function spaces remaining the same, we prove the following set of estimates for solutions to \eqref{eq:KdVmKdVEquation} provided that $s>0$ for some $\theta>0$ and $\delta > 0$:
\begin{equation*}
\left\{\begin{array}{cl}
\Vert u \Vert_{F^s(T)} &\lesssim \Vert u \Vert_{E^s(T)} + \Vert u \partial_x u \Vert_{N^s(T)} + \Vert u^2 \partial_x u \Vert_{N^s(T)} \\
\Vert u \partial_x u \Vert_{N^s(T)} &\lesssim T^\theta \Vert u \Vert_{F^s(T)}^2 \\
\Vert \mathfrak{N}(u) \Vert_{N^s(T)} &\lesssim T^\theta \Vert u \Vert_{F^s(T)}^3 \\
\Vert u \Vert_{E^s(T)}^2 &\lesssim \Vert u_0 \Vert_{H^s}^2 + T^{\theta} M^{c(s)} (\Vert u \Vert_{F^{s-\delta}(T)}^3 + \Vert u \Vert_{F^{s-\delta}(T)}^4) \\
&\quad + M^{-d(s)} ( \Vert u \Vert_{F^{s-\delta}(T)}^3 + \Vert u \Vert_{F^{s-\delta}(T)}^4)\\
&\quad + T^\theta (\Vert u \Vert^4_{F^{s-\delta}(T)} + \Vert u \Vert^5_{F^{s-\delta}(T)} + \Vert u \Vert_{F^{s-\delta}(T)}^6)
\end{array} \right.
\end{equation*}
Above and for the remainder of this subsection, the time localization $T=T(N)=N^{-1}$, i.e., $\alpha=1$, is suppressed in the notation.

The linear estimate follows again from properties of the function spaces, the second nonlinear estimate has been proved in Section \ref{section:shorttimeTrilinearEstimates}. We have to prove the first nonlinear estimate and the extended energy estimate. Below, we assume that $\int_{\T} u dx = 0$. Since the mean is a conserved quantity of the flow, there is no loss of generality. After establishing the above set of estimates, the proof of a priori estimates and existence of solutions for \eqref{eq:KdVmKdVEquation} as stated in Theorem \ref{thm:localExistenceExtension} follows along the lines of Section \ref{section:Proof}. The details are omitted to avoid repitition.

For the nonlinear estimate we have the following lemma.
\begin{lemma}
Let $s>-1/4$. Then, the following estimate holds:
\begin{equation*}
\Vert u \partial_x u \Vert_{N^s(T)} \lesssim T^\theta \Vert u \Vert^2_{F^s(T)}.
\end{equation*}
\end{lemma}
\begin{proof}
By the reductions and notation from Section \ref{section:MultilinearEstimates}, we have to estimate
\begin{equation}
\label{eq:ReductionShorttimeKdVInteraction}
2^{k_3} \sum_{j_3 \geq k_3} 2^{-j_3/2} \Vert 1_{D_{k_3,\leq j_3}} (f_{k_1,j_1} * f_{k_2,j_2}) \Vert_{L^2}
\end{equation}
for different constellations of $k_i,j_i$.\\
\underline{Case A:} Suppose that $|k_3-k_1| \leq 5$, $k_2 \leq k_1 +10$. By function space properties we can suppose that $j_i \geq k_3$. The resonance for a bilinear interaction is given by
\begin{equation*}
\Omega(n_1,n_2,n_3) = n_1^3 + n_2^3 + n_3^3 = - 3 n_1 n_2 n_3, \quad n_1 + n_2 + n_3 = 0.
\end{equation*}
Hence, there is $j_i \geq 2k_3 + k_2 - 5$.

Suppose that $j_3 \geq 2k_3+k_2-5$. Then, we use two $L^4_{t,x}$-Strichartz estimates on $f_{k_2,j_2}$ to find
\begin{equation*}
\begin{split}
\eqref{eq:ReductionShorttimeKdVInteraction} &\lesssim 2^{k_3} \sum_{j_3 \geq 2k_3+k_2} 2^{-j_3/2} \prod_{i=1}^2 2^{j_i/3} \Vert f_{k_i,j_i} \Vert_{L^2} \\
&\lesssim 2^{-k_2/2} 2^{-k_3/3} \prod_{i=1}^2 2^{j_i/2} \Vert f_{k_i,j_i} \Vert_{L^2}.
\end{split}
\end{equation*}
The argument also applies in case $j_1 \geq 2k_3 + k_2 -5$ or $j_2 \geq 2k_3 + k_2 - 5$ by virtue of duality. This proves the claim for the considered interaction provided that $s>-5/6$ due to the slack in the modulation variable in the above display.

\underline{Case B:} Suppose that $k_3 \leq k_1 - 5$. After introducing additional time localization, \eqref{eq:ReductionShorttimeKdVInteraction} is dominated by
\begin{equation*}
2^{k_1} \sum_{j_3 \geq k_3} 2^{-j_3/2} \Vert 1_{D_{k_3,\leq j_3}} (f_{k_1,j_1} * f_{k_2,j_2}) \Vert_{L^2},
\end{equation*}
where $j_i \geq k_1$ for $i=1,2$.\\
First, suppose that $j_3 \geq 2k_1+k_3-5$. Then, we find by applying the bilinear estimate
\begin{equation*}
2^{k_1} \sum_{j_3 \geq 2k_1 + k_3 - 5} 2^{-j_3/2} \Vert 1_{D_{k_3,\leq j_3}}(f_{k_1,j_1} * f_{k_2,j_2}) \Vert_{L^2} \lesssim 2^{-k_3/2} 2^{-k_1/2} \prod_{i=1}^2 2^{j_i/2} \Vert f_{k_i,j_i} \Vert_2.
\end{equation*}
For smaller $j_3$ we have to have $j_i \geq 2k_1 + k_3 - 5$ for $i=1$ or $i=2$. And by duality and an application of the bilinear estimate, we find
\begin{equation*}
\begin{split}
&\quad \quad 2^{k_1} \sum_{k_3 \leq j_3 \leq 2k_1 + k_3} 2^{-j_3/2} \Vert 1_{D_{k_3,\leq j_3}} (f_{k_1,j_1} * f_{k_2,j_2}) \Vert_{L^2}\\
&\lesssim (2k_1) 2^{k_1} 2^{-k_1-k_3/2} 2^{-k_1/2} \prod_{i=1}^2 \Vert f_{k_i,j_i} \Vert_2.
\end{split}
\end{equation*}
These are the key estimates to prove the claim via the general arguments given in detail in Section \ref{section:shorttimeTrilinearEstimates}. 
\end{proof}
The energy estimate will be more involved. We omit the technical considerations not integrating by parts the whole expression and focus on the key estimates.
\begin{lemma}
Let $T \in (0,1]$, $s > 0$ and $a \in S^s_\varepsilon$. Then, there is $\varepsilon(s)>0$ and $\theta(s)>0$ so that for a smooth solution $u$ to \eqref{eq:KdVmKdVEquation} the following estimate holds:
\begin{equation*}
\begin{split}
\Vert u(T) \Vert_{H^a}^2 &\lesssim \Vert u_0 \Vert_{H^a}^2 + T^{\theta} M^{c(s)} (\Vert u \Vert_{F^{s}(T)}^3 + \Vert u \Vert_{F^s(T)}^4) \\
&\quad + M^{-d(s)} ( \Vert u \Vert_{F^{s}(T)}^3 + \Vert u \Vert_{F^s(T)}^4) \\
&\quad + T^\theta (\Vert u \Vert^4_{F^s(T)} + \Vert u \Vert^5_{F^s(T)} + \Vert u \Vert_{F^{s}(T)}^6)
\end{split}
\end{equation*}
\end{lemma}
\begin{proof}
Invoking the fundamental theorem of calculus, we find for the evolution of the $H^a$-norm
\begin{equation*}
\begin{split}
\Vert u(T) \Vert_{H^a}^2 &\lesssim \Vert u_0 \Vert_{H^a}^2 + R^{s,a}_4(T,u) \\
&\quad + | \int_0^T dt \sum_{\substack{n_1+n_2+n_3=0,\\n_i \neq 0}} (\sum_{i=1}^3 a(n_i) n_i ) \hat{u}(t,n_1) \hat{u}(t,n_2) \hat{u}(t,n_3) |.
\end{split}
\end{equation*}

The contribution of the second line will be denoted by $II$. An integration by parts gives
\begin{equation*}
II \lesssim |B^{s,a}_3| + |II_a| + |II_b|,
\end{equation*}
where
\begin{equation*}
\begin{split}
B^{s,a}_3(T,u) &= \left[ \sum_{\substack{n_1+n_2+n_3=0,\\ n_i \neq 0}} \frac{a(n_1) n_1 + a(n_2) n_2 + a(n_3) n_3}{n_1 n_2 n_3} \hat{u}(t,n_1) \hat{u}(t,n_2) \hat{u}(t,n_3) \right]_{0}^T, \\
II_1(T,u) &= \int_0^T dt \sum_{\substack{n_1+n_2+n_3=0,\\ n_i \neq 0}} ( \frac{a(n_1) n_1 + a(n_2) n_2 + a(n_3) n_3}{n_1 n_2 n_3} ) \hat{u}(t,n_1) \hat{u}(t,n_2) \\
&\quad \quad n_3 \sum_{\substack{n_3 = n_4 + n_5,\\ n_4,n_5 \neq 0}} \hat{u}(t,n_4) \hat{u}(t,n_5), \\
II_2(T,u) &= \int_0^T dt \sum_{\substack{n_1+n_2+n_3=0,\\ n_i \neq 0}} ( \frac{a(n_1) n_1 + a(n_2) n_2 + a(n_3) n_3}{n_1 n_2 n_3} \hat{u}(t,n_1) \hat{u}(t,n_2) ) \\
&\quad \quad n_3 ( |\hat{u}(t,n_3)|^2 \hat{u}(t,n_3) + \sum_{\substack{n_3 = n_4 + n_5 + n_6, \\ (*)}} \hat{u}(t,n_4) \hat{u}(t,n_5) \hat{u}(t,n_6) ).
\end{split}
\end{equation*}

Estimate of $B^{s,a}_3$: We find for $|n_i| \sim K_i$, $K_2 \lesssim K_1$ and $t \in \{ 0,T \}$ by an application of Cauchy-Schwarz inequality
\begin{equation*}
\sum_{\substack{n_1+n_2+n_3=0,\\ n_i \neq 0}} \frac{a(n_1)}{n_3^2} \prod_{i=1}^3 \hat{u}(t,n_i) \lesssim \frac{a(K_1)}{K_1^2} K_2^{1/2} \prod_{i=1}^3 \Vert P_{k_i} u \Vert_{F_{k_i}}.
\end{equation*}
This is enough to estimate the boundary term.

Estimate of $II_1$: Without loss of generality suppose that $K_1 \sim K_3 \gtrsim K_2$. It is straight-forward to verify that
\begin{equation*}
\left| \frac{a(n_1)n_1 + a(n_2) n_2 + a(n_3) n_3}{n_1 n_2 (n_1+n_2)} \right| \lesssim \frac{a(N_3)}{N_3^2}.
\end{equation*}
The terms stemming from the deviation of the mKdV-evolution are estimated via bilinear estimates and Bernstein's inequality.

\underline{Case A:} $N_1 \sim N_4 \gtrsim N_5,\; N_2$. Taking into account the symbol size, the derivative from differentiation by parts and the localization in time, two applications of estimate \eqref{eq:BilinearEstimate} on the products $(\hat{u}(n_1) \hat{u}(n_5))$, $(\hat{u}(n_2) \hat{u}(n_4))$ give for dyadic blocks, where $|n_i| \sim K_i$ and
\begin{equation*}
II_1(N_1,N_2,N_4,N_5) \lesssim T \frac{a(N_1)}{N_1} \prod_{i=1}^4 \Vert P_{k_i} u \Vert_{F_{k_i}}.
\end{equation*}
\underline{Case B:} $N_4 \sim N_5 \gg N_1 \gtrsim N_2$. Here, the time localization gives for a dyadic block together with two bilinear estimates on the products $(\hat{u}(n_4) \hat{u}(n_1))$, $(\hat{u}(n_5) \hat{u}(n_2))$ 
\begin{equation*}
II_1(K_1,K_2,K_4,K_5) \lesssim T \frac{a(K_1)}{K_1} \prod_{i=1}^4 \Vert P_{k_i} u \Vert_{F_{k_i}}.
\end{equation*}
In both cases we find after summing over $N_i$ and taking into account the relations between the frequencies
\begin{equation*}
II_1 \lesssim T \Vert u \Vert^4_{F^0(T)}.
\end{equation*}
Estimate of $II_2$: For this estimate we can suppose $K_4 \lesssim K_5 \lesssim K_6$ by symmetry.

\underline{Case A:} $K_1 \sim K_6 \gtrsim K_2, K_5 \gtrsim K_4$. Suppose $K_2 = \min(K_i)$. With two bilinear estimates on $(\hat{u}(n_1) \hat{u}(n_4))$, $(\hat{u}(n_5) \hat{u}(n_6))$ and an application of Bernstein's inequality on $\hat{u}(n_2)$, we find for a dyadic block
\begin{equation*}
II_2(K_1,K_2,K_4,K_5,K_6) \lesssim \frac{K_2^{1/2} a(K_1)}{K_1} \prod_{\substack{i=1,\\i \neq 3}}^6 \Vert P_{k_i} u \Vert_{F_{k_i}}.
\end{equation*}

\underline{Case B:} $K_6 \sim K_5 \gg K_1$. Though the time localization is different, the above argument still yields
\begin{equation*}
II_2(K_1,K_2,K_4,K_5,K_6) \lesssim \frac{a(K_1) K_2^{1/2}}{K_1} \prod_{\substack{i=1,\\ i \neq 3}}^6 \Vert P_{k_i} u \Vert_{F_{k_i}}.
\end{equation*}
Summation over dyadic blocks yields
\begin{equation*}
II_2 \lesssim T \Vert u \Vert^5_{F^0(T)}.
\end{equation*}

Considering KdV-mKdV-evolution, the differentiation by parts of $R_4$ gives in addition to the mKdV-boundary- and mKdV-remainder-term rise to the following expression
\begin{equation*}
\begin{split}
&\int_0^T \sum_{\substack{n_1+n_2+n_3+n_4 = 0,\\ (*)}} \frac{a(n_1) n_1 + a(n_2) n_2 + a(n_3) n_3 + a(n_4)n_4}{3(n_1+n_2)(n_1+n_3)(n_2+n_3)} \hat{u}(t,n_1) \hat{u}(t,n_2) \hat{u}(t,n_3) \\
&\quad \quad ( n_4 \sum_{\substack{n_4 = n_5+n_6, \\ n_i \neq 0}} \hat{u}(t,n_5) \hat{u}(t,n_6) ) dt.
\end{split} 
\end{equation*}

To estimate the above display, we recall that the symbol size is estimated by $\frac{\tilde{a}(2^{n_1})}{2^{2n_1}}$ (cf. Lemma \ref{lem:pointwiseMultiplierBound}). By symmetry we may assume $K_5 \lesssim K_6$.

\underline{Case A:} $K_6 \sim K_1$. Taking the time localization into account and applying two bilinear Strichartz estimates on $(\hat{u}(n_1) \hat{u}(n_2))$, $(\hat{u}(n_5) \hat{u}(n_6))$ and Bernstein's inequality, we find
\begin{equation*}
III_1(K_1,K_2,K_3,K_4,K_6) \lesssim \frac{a(K_1)}{K_1} K_3^{1/2} \prod_{\substack{i=1,\\ i \neq 5}}^6 \Vert P_{k_i} u \Vert_{F_{k_i}}.
\end{equation*}
\underline{Case B:} $K_5 \sim K_6 \gg K_4 \sim K_1 \gtrsim K_2 \gtrsim K_3$: Here, the time localization is different, but the argument is the same to prove
\begin{equation*}
III_1(K_1,K_2,K_3,K_5,K_6) \lesssim \frac{a(K_1) K_3^{1/2}}{K_1} \prod_{\substack{i=1,\\ i \neq 4}}^6 \Vert P_{k_i} u \Vert_{F_{k_i}}.
\end{equation*}
In both cases the summation to find $\lesssim T \Vert u \Vert_{F^0(T)}^5$ is easy. 
\end{proof}
\subsection{mKdV-mKdV-system}
To employ the mKdV-argument to prove existence of solutions and a priori estimates for $s>0$ for small initial data, the following set of estimates (up to the complication from the boundary terms in the energy estimates) is proved for the evolution of $u$:
\begin{equation*}
\left\{\begin{array}{cl}
\Vert u \Vert_{F^s(T)} &\lesssim \Vert u \Vert_{E^s(T)} + \Vert \partial_x (u v^2) \Vert_{N^s(T)} \\
\Vert \partial_x (u v^2) \Vert_{N^s(T)} &\lesssim \Vert u \Vert_{F^s(T)} \Vert v \Vert^2_{F^s(T)} \\
\Vert u \Vert_{E^s(T)}^2  &\lesssim \Vert u_0 \Vert_{H^s}^2 + \Vert v_0 \Vert_{H^s}^2 + T^\theta \Vert u \Vert_{F^{s-\delta}(T)}^2 \Vert v \Vert_{F^{s-\delta}(T)}^2 \\
&\quad + T^\theta \Vert u \Vert^2_{F^{s-\delta}(T)} \Vert v \Vert_{F^{s-\delta}(T)}^2 ( \Vert u \Vert^2_{F^{s-\delta}(T)} + \Vert v \Vert_{F^{s-\delta}(T)}^2)
\end{array}\right.
\end{equation*}
The same set of estimates holds for $v$ mutatis mutandis.

Whereat for the mKdV-equation, trivial resonances could be removed through renormalization, this is not easily possible for \eqref{eq:mKdVmKdVSystem}.\\
However, the trivial resonances $(\partial_x u)(\int uv)$, $(\partial_x v) (\int uv)$ are estimated via the time localization.
\begin{lemma}
\label{lem:TrivialResonances}
Let $s>0$. Then, we find the following estimate to hold:
\begin{equation*}
\Vert (\partial_x u) (\int uv) \Vert_{N^s(T)} \lesssim \Vert u \Vert^2_{F^s(T)} \Vert v \Vert_{F^s(T)}
\end{equation*}
\end{lemma}
\begin{proof}
In Fourier variables, we have
\begin{equation*}
\mathcal{F}_x (\partial_x u \int uv)(n)= in \hat{u}(n) \sum_{n_2} \hat{u}(n_2) \hat{v}(-n_2).
\end{equation*}
We divide the sum over $n$ and $n_2$ up into dyadic blocks $K$, $K_2$, respectively, so that $|n| \sim K$, $|n_2| \sim K_2$. Then, using the reductions and notation from Section \ref{section:EnergyEstimates}, we find for $K_2 \lesssim K$
\begin{equation*}
\begin{split}
&\quad K \sum_{j_4 \geq \log_2(K)} 2^{-j_4/2} \Vert 1_{D_{k_4,j_4}}(\tau_4,n) \int_{\tau_1+\tau_2+\tau_3=\tau_4} f_{k,j_1}(\tau_1,n) \\
&\quad \quad \sum_{|n_2| \sim K_2} f_{k_2,j_2}(\tau_2,n_2) f_{k_3,j_3}(\tau_3,-n_2) \Vert_{L^2_{\tau_4} \ell^2_n} \\
&\lesssim \prod_{i=1}^3 2^{j_i/2} \Vert f_{k_i,j_i} \Vert_{L^2}.
\end{split}
\end{equation*}

In case $K_2 \gg K$, we have to add localization in time to adjust to $K_2$ and have to estimate
\begin{equation*}
\begin{split}
&K_2 \sum_{j_4 \geq \log_2(K)} 2^{-j_4/2} \Vert 1_{D_{k_4,j_4}}(\tau_4,n) \int_{\tau_1+\tau_2+\tau_3=\tau_4} f_{k,j_1}(\tau_1,n) \\
&\quad \quad \sum_{|n_2| \sim K_2} f_{k_2,j_2}(\tau_2,n_2) f_{k_3,j_3}(\tau_3,-n_2) \Vert_{L^2_{\tau_4} \ell^2_n}.
\end{split}
\end{equation*}
For $2^{j_4} \geq K_2$ we use the argument from above. For the sum $K \leq 2^{j_4} \leq K_2$, we use duality to argue that
\begin{equation*}
K_2 \sum_{K \leq 2^{j_3} \leq K_2} 2^{-j_3/2} 2^{j_3/2} 2^{j_1/2} \prod_{i=1}^3 \Vert f_{k_i,j_i} \Vert_{L^2} \lesssim_\varepsilon K_2^\varepsilon \prod_{i=1}^3 2^{j_i/2} \Vert f_{k_i,j_i} \Vert_2.
\end{equation*}

This completes the proof.
\end{proof}
For the energy estimate we use the fundamental theorem of calculus and symmetrization to compute
\begin{equation*}
\begin{split}
&\Vert u(T) \Vert^2_{H^a} + \Vert v(T) \Vert^2_{H^a} = \Vert u_0 \Vert_{H^a}^2 + \Vert v_0 \Vert^2_{H^a} \\
&+ C \int_0^T dt \sum_{0=k_1+\ldots+k_4} \left( \sum_{i=1}^4 a(k_i) k_i \right) \hat{u}(t,k_1) \hat{u}(t,k_2) \hat{v}(t,k_3) \hat{v}(t,k_4) .
\end{split}
\end{equation*}

Note that trivial resonances are cancelled in the sum. Thereafter, the latter expression is amenable to the analysis from Section \ref{section:EnergyEstimates} up to additional resonances, which come up after the integration by parts. These are estimated like in Lemma \ref{lem:TrivialResonances}. This finishes the analysis of \eqref{eq:mKdVmKdVSystem} for small initial data. For large initial data one can argue by rescaling the torus (cf. \cite{Molinet2012}).

\section*{Acknowledgements}
Financial support by the German Research Foundation (IRTG 2235) is gratefully acknowledged. I would like to thank the anonymous referee for a careful reading of an earlier manuscript, which gave rise to many improvements.

\end{document}